\documentclass[smallextended,referee,envcountsect]{svjour3}
\smartqed

\newcommand{\Id}{\mathrm{Id}}

\usepackage{amsmath, amssymb, amsfonts, mathtools}
\usepackage{bm}
\usepackage{tabularx,ragged2e,booktabs,multirow}
\usepackage{extarrows}
\DeclareMathOperator*{\argmin}{arg\,min}

\DeclareMathOperator{\vc}{vec}

\newcommand{\prox}{\mathrm{prox}}
\newcommand{\norm}[1]{\left\lVert#1\right\rVert}
\let\originalleft\left
\let\originalright\right
\renewcommand{\left}{\mathopen{}\mathclose\bgroup\originalleft}
\renewcommand{\right}{\aftergroup\egroup\originalright}

\newcolumntype{L}[1]{>{\raggedright\let\newline\\\arraybackslash\hspace{0pt}}m{#1}}
\newcolumntype{C}[1]{>{\centering\let\newline\\\arraybackslash\hspace{0pt}}m{#1}}
\newcolumntype{R}[1]{>{\raggedleft\let\newline\\\arraybackslash\hspace{0pt}}m{#1}}

\makeatletter
\newcommand{\vast}{\bBigg@{4}}
\makeatother
\def\*#1{\bm{\mathbf{#1}}}

\newcommand{\set}[2]{\left\{ #1\ \left| \ #2 \right. \right\}}

\newcommand{\eqdef}{:=}

\usepackage{xcolor,calc}

\makeatletter
\newcommand{\proofpart}[2]{%
  \par
  \addvspace{\medskipamount}%
  \noindent\emph{Part #1: #2}\par\nobreak
  \addvspace{\smallskipamount}%
  \@afterheading
}
\makeatother
\usepackage{enumerate}
\usepackage{graphicx} % Required to insert images
\DeclareMathOperator{\cl}{cl}
\DeclareMathOperator{\conv}{conv}
\usepackage{multicol}
\usepackage{listings}
\usepackage{stackengine}
\usepackage{csvsimple}
\usepackage[ruled, linesnumbered]{algorithm2e}
\newcommand{\newcontent}[1]{#1}
\usepackage{siunitx}
\usepackage{longtable}
\usepackage{lscape}
\usepackage{csvsimple}
\usepackage{caption}
\usepackage{setspace}

\begin{document}

\title{Efficient Semidefinite Programming \\ with approximate ADMM}

\author{Nikitas Rontsis \and Paul Goulart \and \\ Yuji Nakatsukasa}

\institute{
    Nikitas Rontsis\textsuperscript{1}, Corresponding Author: nrontsis@gmail.com \\
    Paul Goulart\textsuperscript{1}: paul.goulart@eng.ox.ac.uk \\
    Yuji Nakatsukasa\textsuperscript{2}: yuji.nakatsukasa@maths.ox.ac.uk
    \\
    \textsuperscript{1} Department of Engineering Science, University of Oxford, Oxford, UK.\\
    \textsuperscript{2} Mathematical Institute, University of Oxford, Oxford, UK.
}

\maketitle

\begin{abstract}
    Tenfold improvements in computation speed can be brought to the alternating direction method of multipliers (ADMM) for Semidefinite Programming with virtually no decrease in robustness and provable convergence simply by projecting approximately to the Semidefinite cone. Instead of computing the projections via ``exact'' eigendecompositions that scale cubically with the matrix size and cannot be warm-started, we suggest using state-of-the-art factorization-free, approximate eigensolvers\newcontent{,} thus achieving almost quadratic scaling and the crucial ability of warm-starting. Using a recent result from \cite{Goulart2020}, we are able to circumvent the numerical instability of the eigendecomposition and thus maintain tight control on the projection accuracy.  This in turn guarantees convergence, either to a solution or a certificate of infeasibility, of the ADMM algorithm. To achieve this, we extend recent results from \cite{osqpinfeasibility} to prove that reliable infeasibility detection can be performed with ADMM even in the presence of approximation errors. In all of the considered problems of SDPLIB that ``exact'' ADMM can solve in a few thousand iterations, our approach brings a significant, up to 20x, speedup without a noticeable increase on ADMM's iterations.
\end{abstract}

\section{Introduction}
Semidefinite Programming is of central importance in many scientific fields. Areas as diverse as kernel-based learning \cite{Lanckriet2004}, dimensionality reduction \cite{Aspremont2005} analysis and synthesis of state feedback policies of linear dynamical systems \cite{Boyd1994}, sum of squares programming \cite{Prajna2002}, optimal power flow problems \cite{Lavaei2012} and fluid mechanics \cite{Goulart2012} rely on Semidefinite Programming as a crucial enabling technology.
% Another example
% Exact Matrix Completion via Convex Optimization

The wide adoption of Semidefinite Programming was facilitated by reliable algorithms that can solve semidefinite problems with polynomial worst-case complexity \cite{Boyd1994}. For small to medium sized problems, it is widely accepted that primal-dual Interior Point methods are efficient and robust and are therefore often the method of choice. Several open-source solvers, like \texttt{SDPT3} \cite{SDPT3} and \texttt{SDPA} \cite{SDPA}, as well as the commercial solver \texttt{MOSEK} \cite{MosekPy} exist that follow this approach. However, the limitations of interior point methods become evident in large-scale problems, since each iteration requires factorizations of large Hessian matrices. First-order methods avoid this bottleneck and thereby scale better to large problems, with the ability to provide modest-accuracy solutions for many large scale problems of practical interest.

We will focus on the Alternating Directions Method of Multipliers (ADMM), a popular first-order algorithm that has been the method of choice for several popular optimization solvers both for Semidefinite Programming \cite{ODonoghue2016}, \cite{Zheng2017}, \cite{Garstka2019} and other types of convex optimization problems such as Quadratic Programming (QP) \cite{osqp}. Following an initial factorization of an $m \times m$ matrix, every iteration of ADMM entails the solution of a linear system via forward/backward substitution and a projection to the Semidefinite Cone.  For SDPs, this projection operation \newcontent{typically} takes the majority of the solution time, sometimes 90\% or more. Thus, reducing the per-iteration time of ADMM is directly linked to computing conic projections in a time-efficient manner.

The projection of a symmetric matrix $n \times n$ matrix $A$ to the Semidefinite Cone is defined as $$\newcontent{\Pi_{\mathbb{S}_+}(A)} \eqdef \argmin_{X} \norm{A - X}_F,$$ and can be computed in ``closed form'' as a function of the eigendecomposition of $X$. Indeed, assuming
\newcontent{
\begin{equation} \label{admm:eqn:eigendecomposition}
    \begin{bmatrix}V_+ & V_-\end{bmatrix}
    \begin{bmatrix}\Lambda_+ & \\ & \Lambda_-\end{bmatrix}
    \begin{bmatrix}V_+ & V_-\end{bmatrix}^\newcontent{T} \eqdef X
\end{equation}
}
where $V_+$ (respectively $V_-$) is an orthonormal matrix containing the positive (\newcontent{nonpositive}) eigenvectors, and $\Lambda_+$ $(\Lambda_-)$ is a diagonal matrix that contains the \newcontent{respective} positive (nonnegative) eigenvalues of $A$, then
\begin{equation} \label{admm:eqn:projection}
\Pi_{\mathbb{S}_+}(A) = V_+ \Lambda_+ V_+^T = A - V_- \Lambda_- V_-^T.
\end{equation}
The computation of $\Pi_{\mathbb{S}_+}$ therefore entails the (partial) eigendecomposition of $A$ followed by a scaled matrix-matrix product.

The majority of optimization solvers, e.g. \texttt{SCS} \cite{ODonoghue2016} and \texttt{COSMO.jl} \cite{Garstka2019}, calculate $\Pi_{\mathbb{S}_+}$ by computing the full eigendecomposition using \texttt{LAPACK}'s \texttt{syevr} routine\footnote{Detailed in https://software.intel.com/mkl-developer-reference-c-syevr}. There are two important limitations associated with computing full eigendecompositions. Namely, eigendecomposition has cubic complexity with respect to the matrix size $n$ \cite[\S8]{Golub2013}, and it cannot be warm started. This has prompted research on methods for the approximate computation of a few eigenpairs in an iterative fashion \cite{Saad2011}, \cite{Bai2000}, \cite{Parlett1998}, and the associated development of relevant software tools such as the widely used \texttt{ARPACK} \cite{Lehoucq1998},
% shipped by default with  \texttt{MATLAB}, Julia and Scipy
and the more recent
\texttt{BLOPEX} \cite{Knyazev2001} and \texttt{PRIMME} \cite{Stathopoulos2010}. The reader can find surveys of relevant software in \cite{Hernandez2009} and \cite[\S 2]{Stathopoulos2010}

% Perhaps surprisignly
However, the use of iterative eigensolvers in the Semidefinite optimization community has been very limited. To the best of our knowledge, \newcontent{ the use of approximate eigensolvers has been limited in widely-available ADMM implementations. In a related work, \cite{Li2020} considered the use of polynomial subspace extraction to avoid expensive eigenvalue decompositions required by first order-methods (including ADMM) and showed improved performance in problems of low-rank structure, while still maintaining convergence guarantees.} In the wider area of first-order methods, \cite[\S 3.1]{Wen2010} considered \texttt{ARPACK} but disregarded it on the basis that it does not allow efficient warm starting, suggesting that it should only be used when the problem is known a priori to have low rank. .
Wen's suggestion of using \texttt{ARPACK} for SDPs whose solution are expected to be low rank has been demonstrated recently by \cite{Souto2018}. At every iteration, \cite{Souto2018} uses \texttt{ARPACK} to compute the $r$ largest eigenvalues/vectors and then  uses the approximate projection $\tilde \Pi(A) = \sum_{i=1}^r \max(\lambda_i, 0) v_i v_i^T$. The projection error can then be bounded by $$\norm{\Pi(A) - \tilde \Pi(A)}_F^2 = \norm{\sum_{i = r+1}^n \max(\lambda_i, 0) v_i v_i^T}_F^2 \leq (n-r)\max(\lambda_r, 0)^2.$$ The parameter $r$ is chosen such that it decreases with increasing iteration count so that the projection errors are summable.
% across iterations.
The summability of the projection errors is important, as it has been shown to ensure convergence of averaged non-expansive operators \cite[Proposition 5.34]{Bauschke2017} and for ADMM in particular \cite[Theorem 8]{Eckstein1992}.

However, the analysis of \cite{Souto2018} depends on the assumption that the iterative eigensolver will indeed compute the $r$ largest eigenpairs ``exactly''. This is both practically and theoretically problematic; the computation of eigenvectors is numerically unstable since it depends inverse-proportionally on the spectral gap (defined as the distance between the corresponding eigenvalue and its nearest neighboring eigenvalue; refer to \S\ref{admm:subsec:termination} for a concise definition), and therefore no useful bounds can be given when repeated eigenvalues exist.

In contrast, our approach relies on a novel bound that characterizes the projection accuracy independently of the spectral gaps, depending only on the residual norms. The derived bounds do not require that the eigenpairs have been computed ``exactly'', but hold for \emph{any} set of approximate eigenpairs obtained via the Rayleigh-Ritz process. This allows us to compute the eigenpairs with a relatively loose tolerance while still retaining convergence guarantees. Furthermore, unlike \cite{Souto2018}, our approach has the ability of warm-starting of the eigensolver, which typically results in improve computational efficiency.

On the theoretical side, we extend recent results regarding the detection of primal or dual infeasibility. It is well known that if an SDP problem is infeasible then the iterates of ADMM will diverge \cite{Eckstein1992}. This is true even when the iterates of ADMM are computed approximately with summable approximation errors. Hence, infeasibility can be detected in principle by stopping the ADMM algorithm when the iterates exceed a certain bound.  However, this is unreliable both in practice, because it depends on the choice of the bound, and in theory, because it does not provide certificates of infeasibility \cite{Boyd2004}. Recently, \cite{osqpinfeasibility} has shown that the successive differences of ADMM's iterates, which always converge regardless of feasibility, can be used to reliably detect infeasibility and construct infeasibility certificates. This approach has been used successfully in the optimization solver \texttt{OSQP} \cite{osqp}. We extend Banjac's results to show that they hold even when ADMM's iterates are computed approximately, under the assumption that the approximation errors are summable.

\emph{Notation used}: Let $\mathcal{H}$ denote a real Hilbert space equipped with an inner-product induced norm $\norm{\cdot} = \newcontent{\sqrt{\langle \cdot, \cdot \rangle}}$ and $\text{Cont}(\mathcal{D})$ the set of nonexpansive operators in $\mathcal{D} \subseteq \mathcal{H}$. $\cl{\mathcal{D}}$ denotes the closure of $\mathcal{D}$, $\conv{\mathcal{D}}$ the convex hull of $\mathcal{D}$, and $\mathcal{R}(T)$ the range of $T$. $\Id$ denotes the identity operator on $\mathcal{H}$ while $I$ denotes an identity matrix of appropriate dimensions. For any scalar, nonnegative $\epsilon$, let $x \approx_\epsilon y$ denote the following relation between $x$ and $y$: $\norm{x - y} \leq \epsilon$. $\mathbb{S}_+$ denotes the set of positive semidefinite matrices with a dimension that will be obvious from the context. \newcontent{Finally, define $\Pi_{\mathcal{C}}$ the projection, $\mathcal{C}^\infty$ the recession cone, and $S_\mathcal{C}$ the support function associated with a set $\mathcal{C}$}.

\section{Approximate ADMM} \label{admm:sec:algorithm}
Although the focus of this paper is on Semidefinite Programming, our analysis holds for more general convex optimization problems that allow for combinations of semidefinite Problems, Linear Programs (LPs), Quadratic Programs (QPs), Second Order Cone Programs (SOCPs) among others\footnote{\newcontent{Note that any problem of the form \eqref{admm:eqn:main_problem} can be converted to an SDP by noting that the positive orthant and the second order cone can be expressed as a semidefinite cone, and by considering the epigraph form of \eqref{admm:eqn:main_problem} \cite[\S4.1.3]{Boyd2004}.}}. In particular, the problem form we consider is defined as
\begin{equation}\label{admm:eqn:main_problem} \tag{$\mathcal{P}$}
  \begin{array}{ll}
    \mbox{minimize}   & \frac{1}{2} x^T P x + q^T x \\
    \mbox{subject to} & A x = z \\
                      & z \in \mathcal{C},
  \end{array}
\end{equation}
where $x \in \newcontent{\mathbb{R}^n}$ and $z \in \newcontent{\mathbb{R}^m}$ are the decision variables, $P \in \newcontent{\mathbb{S}^n_{+}}$, $q \in \newcontent{\mathbb{R}^n}$, $A \in \mathbb{R}^{m \times n}$ and $\mathcal{C}$ is \newcontent{a translated composition of the positive orthant, second order and/or semidefinite cones.}

We suggest solving \eqref{admm:eqn:main_problem}, i.e. finding a solution $(\bar x, \bar z, \bar y)$ where $\bar y$ is a Lagrange multiplier for the equality constraint of \eqref{admm:eqn:main_problem}, with the approximate version of ADMM described in Algorithm \ref{admm:alg:admm_approximate}.
\begin{algorithm}
  \textbf{given} initial values $x^0, y^0, z^0$, parameters $\rho > 0, \sigma > 0, \alpha \in (0, 2)$, the summable sequences $(\mu^k)_{k \in \mathbb{N}}, (\nu^k)_{k \in \mathbb{N}}$\, and $x \approx_\epsilon y$ denoting that the vectors $x$, $y$ satisfy $\norm{x - y} \leq \epsilon$;
  \For{$k = 0, \dots$ until convergence}{
    $
    \begin{bmatrix}
      \tilde x^{k+1} \\
      \tilde z^{k + 1}
    \end{bmatrix} \approx_{\mu^k}
    \begin{bmatrix}
      P + \sigma I & \rho A^T \\
      \rho A & -\rho I
    \end{bmatrix}
    \bigg\backslash
    \left(
      \begin{bmatrix}
        \sigma I & \rho A^T \\
        0 & 0
      \end{bmatrix}
      \begin{bmatrix}
        x^k \\
        z^k - y^k/\rho
      \end{bmatrix}
      -
      \begin{bmatrix}
        q \\
        0
      \end{bmatrix}
    \right)$\;
    $x^{k + 1} = \alpha \tilde x^{k + 1} + (1 - \alpha)x^k$ \;
    $z^{k + 1} \approx_{\nu^k} \Pi_\mathcal{C}(\alpha \tilde z^{k + 1} + (1 - \alpha)z^k + y^k/\rho)$\;
    $y^{k+1} = y^{k} + \rho(\alpha \tilde x^{k+1} + (1 - \alpha)z^k - z^{k + 1})$
  }
  \caption{Solving \eqref{admm:eqn:main_problem} with approximate ADMM}
  \label{admm:alg:admm_approximate}
\end{algorithm}
As is \newcontent{common} in the case in ADMM methods, our Algorithm consists of repeated solutions of linear systems (line \newcontent{2}) and projections to $\mathcal{C}$ (line \newcontent{4}). These steps are the primary drivers of efficiency of ADMM and are typically computed to machine precision via matrix factorizations. Indeed, Algorithm \ref{admm:alg:admm_approximate} was first introduced by \cite{osqp} \newcontent{and \cite{osqpinfeasibility}} in the absence of approximation errors. However, ``exact'' computations can be prohibitively expensive for large problems (and indeed impossible in finite-precision arithmetic), and the practitioner may have to rely on approximate methods for their computation.
%, i.e. solve the linear systems approximately and projection to $\mathcal{C}$ approximately.
For example, \cite[\S 4.3]{Boyd2011} suggests using the Conjugate Gradient method for approximately solving the linear systems embedded in ADMM. In Section \ref{admm:chapter:krylov}, we suggest specific methods for the approximation computation of ADMM steps with a focus in the operation of line \newcontent{4}. Before moving into particular methods, we  first discuss the convergence properties of Algorithm~\ref{admm:alg:admm_approximate}.

\newcontent{Our} analysis explicitly accounts for approximation errors and provides convergence guarantees, either to solutions or certificates of infeasibility, in their presence. In general
%, and in the absence of the summability requirement for $(\mu^k)$ and $(\nu^k)$ of Algorithm \ref{admm:alg:admm_approximate},
when ADMM's steps are computed approximately, ADMM might lose its convergence properties. Indeed, when the approximation errors are not controlled appropriately, the Fej\'er monotonicity \cite{Bauschke2017} of the iterates and any convergence rates of ADMM can be lost. In the worst case, the iterates could diverge. However, the following Theorem, which constitutes the main theoretical result of this paper, shows that Algorithm \ref{admm:alg:admm_approximate} converges either to a solution or to a certificate of infeasibility of \eqref{admm:eqn:main_problem} due to the requirement that the approximation errors are summable across the Algorithm's iterations.
\begin{theorem} \label{admm:thm:alg_convergence}
  Consider the iterates $x^k, z^k$, and $y^k$ of Algorithm \ref{admm:alg:admm_approximate}. If a KKT point exists for \eqref{admm:eqn:main_problem}, then \newcontent{$(x^k, z^k, y^k)$} converges to a KKT point, i.e. a solution of \eqref{admm:eqn:main_problem}\newcontent{, when $k \mapsto \infty$}.
  Otherwise, the successive differences
  \begin{equation*}
    \delta x \eqdef \lim_{k \to \infty} x^{k + 1} - x^k, \quad \text{and} \quad
    \delta y \eqdef \lim_{k \to \infty} y^{k + 1} - y^k.
  \end{equation*}
  still converge and can be used to detect infeasibility as follows:
  \begin{enumerate}[(i)]
    \item If $\delta y \neq 0$ then \eqref{admm:eqn:main_problem} is primal infeasible and $\delta y$ is a certificate of primal infeasibility \cite[Proposition 3.1]{osqpinfeasibility} in that it satisfies
    \begin{equation} \label{admm:eqn:primal_infeasibility}
      A^T \delta y = 0 \quad \text{and} \quad S_\mathcal{C}(\delta y) < 0.
    \end{equation}
    \item If $\delta x \neq 0$ then \eqref{admm:eqn:main_problem} is dual infeasible and $\delta x$ is a certificate of dual infeasibility \cite[Proposition 3.1]{osqpinfeasibility} in that it satisfies
    \begin{equation} \label{admm:eqn:dual_infeasibility}
      P\delta x = 0, \quad A \delta x \in \mathcal{C}^{\infty}, \quad \text{and} \quad q^T \delta x < 0.
    \end{equation}
    \item If both $\delta x \neq 0$ and $\delta y \neq 0$ then \eqref{admm:eqn:main_problem} is both primal and dual infeasible and \newcontent{$(\delta x, \delta y)$} are certificates of primal and dual infeasibility as above.
  \end{enumerate}
\end{theorem}
In order to prove  Theorem \ref{admm:thm:alg_convergence} we must first discuss some key properties of ADMM. This will provide the theoretical background that will allow us to present the proof in section \ref{admm:sec:proof}. Then, in section \ref{admm:chapter:krylov} we will discuss particular methods for the approximate computation of ADMM's steps that can lead to significant speedups.

\section{The asymptotic behaviour of approximate ADMM} \label{admm:sec:admm}
In this section we present ADMM in a general setting, express it as an iteration over an averaged operator, and then consider its convergence when this operator is computed only approximately.

ADMM is used to solve split optimization problems of the following form
\begin{equation}\label{admm:eqn:splitting} \tag{$\mathcal{S}$}
  \begin{array}{ll}
    \mbox{minimize}   & f(\chi) + g(\psi) \\
    \mbox{subject to} & \chi = \psi
 \end{array}
\end{equation}
where $\chi, \psi$ denote the decision variables on $\mathbb{R}^\ell$ which is equipped with an inner product induced norm $\norm{\cdot} = \langle \cdot, \cdot \rangle$. The functions $f: \mathbb{R}^\ell \rightarrow [-\infty, +\infty]$, $g: \mathbb{R}^\ell \rightarrow [-\infty, +\infty]$ are proper, lower-semicontinuous, and convex.

ADMM works by alternately minimizing the augmented Lagrangian of \eqref{admm:eqn:splitting}, defined as\footnote{\newcontent{Note that, following \cite{osqpinfeasibility}, this definition can account for the penalty parameters $\rho$ and $\sigma$ of Algorithm \ref{admm:alg:admm_approximate} via an appropriate definition for $\norm{\cdot}$, as done later in \eqref{admm:eqn:norm}.}}
\begin{equation} \label{admm:eqn:lagrangian}
  L(\chi, \psi, \omega) \eqdef f(\chi) + g(\psi) + \langle \omega, \chi - \psi \rangle + \frac{1}{2}\norm{\chi - \psi}^2
\end{equation}
over $\chi$ and $\psi$. That is, ADMM consists of the following iterations
\begin{align}
  \label{admm:eqn:chi_update} \tag{{\sc{admm}}\textsubscript{1}}
  \chi^{k+1} &= \argmin_\chi L(\chi, \psi^k, \omega^k) \\
  \label{admm:eqn:psi_update} \tag{{\sc{admm}}\textsubscript{2}}
  \psi^{k+1} &= \argmin_\psi L(\bar \chi^{k+1}, \psi, \omega^k) \\
  \label{admm:eqn:omega_update} \tag{{\sc{admm}}\textsubscript{3}}
  \omega^{k+1} &= \omega^{k} + (\bar \chi^{k+1} - \psi^{k+1})
\end{align}
where $\bar \chi^{k + 1}$ is a relaxation of $\chi^{k+1}$ with $\bar \chi^{k + 1} = \alpha \chi^{k + 1} + (1 - \alpha) \psi^{k}$ for some relaxation parameter $\alpha \in (0, 2)$.

Although \eqref{admm:eqn:chi_update}-\eqref{admm:eqn:omega_update} are useful for implementing ADMM, theoretical analyses of the algorithm typically consider ADMM as an iteration over an averaged operator. To express ADMM in operator form, note that \eqref{admm:eqn:chi_update} and \eqref{admm:eqn:psi_update} can be expressed in terms of the \emph{proximity operator}~\cite[\S 24]{Bauschke2017}
\begin{equation}
  \prox_f(\phi) \eqdef
  \argmin_{\chi}\left(f(\chi) +
  \frac{1}{2}\norm{\chi - \phi}^2 \right), \\
\end{equation}
and the similarly defined \newcontent{$\prox_g$},
as
\[
  \chi^{k + 1} = \prox_f(\psi^k - \omega^k),\quad \psi^{k + 1} = \prox_g(\bar \chi^{k + 1} + \omega^k),
\] respectively.
Now, using the \emph{reflections} of $\prox_f$ and $\prox_g$, i.e. $R_f \eqdef 2\prox_f - \Id$ and $R_g \eqdef 2\prox_g - \Id$, we can express ADMM as an
iteration over the \mbox{$\frac{1}{2} \alpha$-averaged} operator
\begin{equation}\label{admm:eqn:admm_operator}
  T \eqdef \left(1 - \frac{1}{2}\alpha \right)\Id + \frac{1}{2}\alpha R_f R_g
  % T \eqdef \left(1- \frac{1}{2} \alpha \right) I + \frac{1}{2}\alpha R_g \circ R_f
\end{equation}
on the variable $\phi^k \eqdef \bar \chi^k + \omega^{k - 1}$ (see \cite[\S3.A]{Thesis} or \cite[Appendix B]{Giselsson2016} for details). The variables $\psi, \chi, \omega$ of \eqref{admm:eqn:chi_update}-\eqref{admm:eqn:omega_update} can then be obtained from $\phi$ as
\begin{equation} \label{admm:eqn:relations}
  \chi^{k + 1} = \prox_f R_g \phi^k, \quad
  \psi^k = \prox_g \phi^k, \quad \text{and} \quad \omega^k = (\Id - \prox_g)\phi^k.
\end{equation}

We are interested in the convergence properties of ADMM when the operators $\prox_f, \prox_g$, and thus $T$, are computed inexactly. In particular, we suppose that the iterates are generated as
\begin{align}
  \label{admm:eqn:admm_sequence}
  (\forall k \in \mathbb{N}) \quad \phi^{k+1} =
  \left(1 - \frac{1}{2}\alpha\right)\phi_k + \frac{1}{2}\alpha \left(R_f \left(R_g \phi^k + \epsilon_g^k \right) + \epsilon_f^k\right)
\end{align}
for some error sequences $\epsilon_f^k, \epsilon_g^k \in \mathbb{R}^\ell$. Our convergence results will depend on the assumption that $\norm{\epsilon_f^k}$ and $\norm{\epsilon_g^k}$ are summable. This implies that $\phi^k$ can be considered as an approximate iteration over $T$, i.e.
\begin{equation}
  \phi^{k + 1} \approx_{\epsilon^k} T \phi^k,
\end{equation}
for some summable error sequence $(\epsilon^k)$. Indeed, since $R_g$ and $R_f$ are nonexpansive, we have
\begin{align*}
  \norm{R_f \left(R_g \phi^k + \epsilon_g^k \right) + \epsilon_f^k - R_f R_g \phi^k}
  &\leq \norm{R_g \phi^k + \epsilon_g^k - R_g \phi^k} + \norm{\epsilon_f^k} \\
  &\leq \norm{\epsilon_f^k} + \norm{\epsilon_g^k},
\end{align*}
or $\norm{\phi^{k+1} - T\phi^k} \leq \alpha\norm{\epsilon_f^k}/2 + \alpha\norm{\epsilon_g^k}/2$, from which the summability of $\norm{\epsilon^k}$ follows.

It is well known that, when $\norm{\epsilon_f^k}, \norm{\epsilon_g^k}$ are summable, \eqref{admm:eqn:admm_sequence} converges to a solution of \eqref{admm:eqn:splitting}, obtained by $\phi$ according to \eqref{admm:eqn:relations}, provided that \eqref{admm:eqn:splitting} has a KKT point \cite[Theorem 8]{Eckstein1992}. We will show that, under the summability assumption, $\delta \phi = \lim \phi^{k + 1} - \phi^k$ always converges, \emph{regardless} of whether \eqref{admm:eqn:splitting} has a KKT point:
\begin{theorem} \label{admm:thm:delta_approx}
  The successive differences $lim_{k\rightarrow \infty}(\phi^{k+1}-\phi^{k})$ of \eqref{admm:eqn:admm_sequence} converge to the \newcontent{unique minimum-norm element} of $\cl{\mathcal{R}(\Id - T)}$ provided that $\sum \norm{\epsilon_f^{k}} < \infty$ and $\sum \norm{\epsilon_g^{k}} < \infty$.
\end{theorem}
\begin{proof}
  This is a special case of Proposition \ref{admm:prop:delta_approx} of Appendix \ref{admm:app:delta}.
\end{proof}
Theorem \ref{admm:thm:delta_approx} will prove useful in detecting infeasibility, as we will show in the following section.
\section{Proof of Theorem \ref{admm:thm:alg_convergence}} \label{admm:sec:proof}
We now turn our attention to proving Theorem \ref{admm:thm:alg_convergence}. To this end, note that \eqref{admm:eqn:main_problem} can be regarded as a special case of \eqref{admm:eqn:splitting} \cite{osqpinfeasibility,osqp}. This becomes clear if we set $\chi = (\tilde x, \tilde z), \psi = (x, z),$ and define
%, use the norm $\norm{(x, z)} =  \sigma \norm{x}^2 + \rho\norm{z}^2$
\begin{align}
  \label{admm:eqn:f}
  f(\tilde x, \tilde z) &\eqdef \frac{1}{2} \tilde x^T P \tilde x + q^T \tilde x + \mathcal{I}_{A\tilde x=\tilde z}(\tilde x, \tilde z), \\
  \label{admm:eqn:g}
  g(x, z) &\eqdef \mathcal{I}_{\mathcal{C}}(z),
\end{align}
where $\mathcal{I}_\mathcal{\mathcal{C}}(z)$ denotes the indicator function of $\mathcal{C}$. Furthermore, using the analysis of the previous section and defining the norm
\begin{equation} \label{admm:eqn:norm}
  \norm{(x, z)} = \sqrt{\sigma \norm{x}_2^2 + \rho \norm{z}_2^2}.
\end{equation} we find that Algorithm \ref{admm:alg:admm_approximate} is equivalent to  iteration \eqref{admm:eqn:admm_sequence}.

First, we show that if $\eqref{admm:eqn:main_problem}$ has a KKT point then Algorithm \ref{admm:alg:admm_approximate} converges to its primal-dual solution. Due to \eqref{admm:eqn:f}--\eqref{admm:eqn:norm}, every KKT point $(\bar x, \bar z, \bar y)$ of \eqref{admm:eqn:main_problem} produces a KKT point
\begin{equation} \label{admm:eqn:kkt_relation}
  (\bar \chi, \bar \psi, \bar \omega) = ((\bar x, \bar z), (\bar x, \bar z), (0, \bar y/\rho))
\end{equation}
for \eqref{admm:eqn:splitting}.  Likewise, every KKT point of \eqref{admm:eqn:splitting} is in the form of \eqref{admm:eqn:kkt_relation} (right) and gives a KKT point $(\bar x, \bar z, \bar y)$ for \eqref{admm:eqn:main_problem}. Thus, according to \cite[Theorem 8]{Eckstein1992}, Algorithm \ref{admm:alg:admm_approximate} converges to a KKT point of \eqref{admm:eqn:main_problem}, assuming that a KKT point exists.

It remains to show points $(i)-(iii)$ of Theorem \ref{admm:thm:alg_convergence}.
These are a direct consequence of \cite[Theorem 5.1]{osqpinfeasibility} and the following proposition:
\begin{proposition} \label{admm:prop:limits_relations}
  The following limits
  $$%\frac{x^k}{n}, \quad \frac{y^k}{n}, \quad \frac{z^k}{n},\quad
   \delta x \eqdef \lim_{k \to \infty} x^{k+1} - x^k, \quad \delta y \eqdef \lim_{k \to \infty} y^{k + 1} - y^k, % \quad z^{k+1} - z^k,
  $$
  defined by the iterates of Algorithm \ref{admm:alg:admm_approximate}, converge to the respective limits defined by the iterates of Algorithm \ref{admm:alg:admm_approximate} with $\mu^k = \nu^k = 0 \; \forall k \in \mathbb{N}$.
\end{proposition}
\begin{proof}
  According to \cite[\S 3.A]{Thesis} we can rewrite Algorithm \ref{admm:alg:admm_approximate} as follows
  \begin{subequations}
    \begin{align}
      z^{k} &\approx_{\nu^{k - 1}} \Pi_{\mathcal{C}}(\upsilon^k) \\
      (\tilde x^{k + 1}, \tilde z^{k + 1}) &\approx_{\mu^{k}} \prox_f
      ((x^k, 2z^k - \upsilon^k)) \\
      \label{admm:eqn:expanded_operator_x}
      x^{k + 1} &= x^{k} - \alpha(\tilde x^{k + 1} - x^k)\\
      \label{admm:eqn:expanded_operator_u}
      \upsilon^{k + 1} &= \upsilon^{k} + \alpha(\tilde z^{k + 1} - z^k)
    \end{align}
  \end{subequations}
  where $(x^k, \upsilon^k) \eqdef \phi^k$ and $y^k$ can be obtained as $y^k = \rho(\upsilon^k - z^k)$.

  Define $\delta x^k \eqdef x^{k + 1} - x^k$, $\forall k \in \mathbb{N}$ and $\delta z^k$, $\delta \upsilon^k$, $\delta \tilde x^k$, $\delta \tilde z^k$ in a similar manner. Due to Theorem \ref{admm:thm:delta_approx} and \cite[Lemma 5.1]{osqpinfeasibility} we conclude that $\lim_{k \to \infty} \delta x^k$ and $\lim_{k \to \infty} \delta \upsilon^k$, defined by the iterates of Algorithm \ref{admm:alg:admm_approximate}, converge to the respective limits defined by the iterates of Algorithm \ref{admm:alg:admm_approximate} with $\mu^k = \nu^k = 0 \; \forall k \in \mathbb{N}$.

  To show the same result for $\delta y$, first recall that $y^k = \rho(\upsilon^k - z^k)$.  It then suffices to show the desired result for $\lim_{k \to \infty} \delta z^k$. We show this using  arguments  similar to \cite[Proposition 5.1 (iv)]{osqpinfeasibility}. Indeed, note that due to \eqref{admm:eqn:expanded_operator_x}-\eqref{admm:eqn:expanded_operator_u} we have
  \begin{align*}
    -(\delta x^{k + 1} - \delta x^k)/\alpha = \delta x^k - \delta \tilde x^{k + 1} \\
    -(\delta \upsilon^{k + 1} - \delta \upsilon^k)/\alpha = \delta z^k - \delta \tilde z^{k + 1}
  \end{align*}
  and thus $\lim_{k \to \infty} \delta x^k = \lim_{k \to \infty} \delta \tilde x^k$ and $\lim_{k \to \infty} \delta z^k = \lim_{k \to \infty} \delta \tilde z^k$.
  Furthermore, due to \eqref{admm:eqn:f} we have
  $
  A\tilde x^{k + 1} - \tilde z^{k + 1} = e^k
  $
  for some sequence $(e^k)$ with summable norms, thus
  $$\lim_{k \to \infty} \delta \tilde z^k = A \lim_{k \to \infty} \delta \tilde x^k = A \lim_{k \to \infty} \delta x^k$$ and the claim follows due to \cite[Proposition 5.1 (i) and (iv)]{osqpinfeasibility}.
\end{proof}

\section{Krylov-Subspace Methods for ADMM} \label{admm:chapter:krylov}
In this Section, we suggest suitable methods for calculating  the individual steps of Algorithm \ref{admm:alg:admm_approximate}. We will focus on Semidefinite Programming, i.e., when $\mathcal{C}$ is the semidefinite cone. After an initial presentation of state-of-the-art methods used for solving linear systems approximately, we will describe (in \S \ref{admm:subsec:lobpcg}) LOBPCG, the suggested method for projecting onto the semidefinite cone. Note that some of our presentation recalls established linear algebra techniques that we include for the sake of completeness.

We begin with a discussion of the Conjugate Gradient method, a widely used method for the solution of the linear systems embedded in Algorithm \ref{admm:alg:admm_approximate}. Through CG's presentation we will introduce the Krylov Subspace which is a critical component of LOBPCG. Finally, we will show how we can assure that the approximation errors are summable across ADMM iterations, thus guaranteeing convergence of the algorithm.

The linear systems embedded in Algorithm \ref{admm:alg:admm_approximate} are in the following form
\begin{equation} \label{admm:eqn:kkt_system}
  \underbrace{
  \begin{bmatrix}
    P + \sigma I & \rho A^T \\
    \rho A & -\rho I
  \end{bmatrix}
  }_{\eqdef Q}
\begin{bmatrix}
  \tilde x^{k+1} \\
  \tilde z^{k + 1}
\end{bmatrix} =
\underbrace{
  \begin{bmatrix}
    \sigma I & \rho A^T \\
    0 & 0
  \end{bmatrix}
  \begin{bmatrix}
    x^k \\
    z^k - y^k/\rho
  \end{bmatrix}
  -
  \begin{bmatrix}
    q \\
    0
  \end{bmatrix}}_{\eqdef b^k}.
\end{equation}
The linear system \eqref{admm:eqn:kkt_system} belongs to the widely studied class of symmetric quasidefinite systems \cite{Benzi05}, \cite{Orban2017}. Standard scientific software packages, such as the Intel Math Kernel Library and the Pardiso Linear Solver, implement methods that can solve \eqref{admm:eqn:kkt_system} approximately. Since the approximate solution \eqref{admm:eqn:kkt_system} can be considered standard in the Linear Algebra community, we will only discuss the popular class of Krylov Subspace methods, which includes the celebrated Conjugate Gradient method\footnote{The Conjugate Gradient Method is only suitable for Positive Definite Linear Systems. However, \eqref{admm:eqn:kkt_system} can be solved with CG via a variable reduction which yields a smaller positive linear system \cite[\S1]{Orban2017}.}. Although CG has been used in ADMM extensively [\S 4.3.4]\cite{Boyd2011}, \cite{ODonoghue2016}, its presentation will be useful for introducing some basic concepts that are shared with the main focus of this section, i.e. the approximate projection to the semidefinite cone.

From an optimization perspective, Krylov subspace algorithms for solving linear systems can be considered an improvement of gradient methods. Indeed, solving $Ax = b$, where $A \in \mathbb{S}^n_{++}$ via gradient descent on the objective function $c(x) \eqdef \frac{1}{2}x^T Ax - x^T b$ amounts to the following iteration
\begin{equation} \label{admm:eqn:gd_linear_system}
  (\forall k \in \mathbb{N}) \quad x^{k + 1} = x^k - \beta^k \nabla c(x) =  x^k - \beta^k \underbrace{(Ax^k - b)}_{\eqdef r^k}
\end{equation}
where $\beta^k$ is the step size at iteration $k$. Note that
$$x^{k + 1} \in x_0 + \underbrace{\text{span}(r_0, Ar_0, \cdots A^k r_0)}_{\eqdef \mathcal{K}_k(A, r^0)},$$
where $\mathcal{K}_k(A, r_0)$ is known as the \emph{Krylov Subspace}. As a result, the following algorithm,
\begin{equation} \tag{CG} \label{admm:eqn:cg}
  (\forall k \in \mathbb{N}) \quad x^{k + 1} = \argmin_{x \in x^0 + \mathcal{K}_k(A, r^0)}{\frac{1}{2}x^T Ax - x^T b}
\end{equation}
is guaranteed to yield results that are no worse than gradient descent. What is remarkable is that \eqref{admm:eqn:cg} can be implemented efficiently in the form of two-term recurrences, resulting in the \emph{Conjugate Gradient} (CG) Algorithm \cite[\S11.3]{Golub2013}.

We now turn our attention to the projection to the Semidefinite cone\newcontent{, which we have already defined in the introduction. Recalling \eqref{admm:eqn:projection}, we note that} the projection to the semidefinite cone can be computed via either the positive or the negative eigenpairs of $A$. As we will see, the cost of approximating eigenpairs of a matrix depends on their cardinality, thus computing $\Pi_{\mathbb{S}_+}(A)$ with the positive eigenpairs of $A$ is preferable when $A$ has mostly nonpositive eigenvalues, and vice versa. In the following discussion we will focus on methods that compute the positive eigenpairs of $A$, thus assuming that $A$ has mostly nonpositive eigenvalues. The opposite case can be easily handled by considering $-A$.

Similarly to CG, the class of Krylov Subspace methods is very popular for the computation of ``extreme'' eigenvectors of an $n \times n$ symmetric matrix $A$ and can be considered as an improvement to gradient methods. In the subsequent analysis we will make frequent use of the real eigenvalues of $A$, which we denote with $\lambda_1 \geq \dots \geq \lambda_n$ and a set of corresponding orthogonal eigenvectors $\upsilon_1, \dots \upsilon_n$. The objective to be maximized in this case is the Rayleigh Quotient,
\begin{equation}
  r(x) \eqdef \frac{x^TAx}{x^T x}.
\end{equation}
% due to the Courant characterization of eigenpairs \cite[Theorems 1.10]{Saad2011}:
%$$\lambda_k = r(v_k) = \max_{x \neq 0, x \perp v_1, \dots, v_{k -1}} r(x).$$
due to the fact that the maximum and the minimum values of $r(x)$ are $\lambda_1$ and $\lambda_n$ respectively with $\upsilon_1$ and $\upsilon_n$ as corresponding maximizers \cite[Theorem 8.1.2]{Golub2013}. Thus, we end up with the following gradient ascent iteration
\begin{align} \label{admm:eqn:gradient_ascent}
  (\forall k \in \mathbb{N}) \quad x^{k + 1} &= \alpha^k x^k - \beta^k \nabla r(x^k) \\
  \nonumber
  &= \alpha^k x^k - 2\beta^k\left(Ax^k - r(x^k)x^k \right) % /\norm{x^k}_2^2
  % See https://math.stackexchange.com/questions/2847017/gradient-of-the-rayleigh-quotient
\end{align}
where the ``stepsizes'' $\alpha^k$ and $\beta^k$ and the initial point $x^0$ are chosen so that all the iterates lie on the unit sphere. Although $r(x)$ is nonconvex, \eqref{admm:eqn:gradient_ascent} can be shown to converge when appropriate stepsizes are used. For example, if we choose $\alpha_k = -2\beta^k r(x^k) \Rightarrow x^{k + 1} \propto Ax^k$ $\forall k \in \mathbb{N}$, then \eqref{admm:eqn:gradient_ascent} is simply the \emph{Power Method}, which is known to converge linearly to an eigenvector associated with $\max|\lambda_i|$. Other stepsize choices can also assure convergence to an eigenvector associated with $\max \lambda_i$ \cite[11.3.4]{Bai2000}, \cite[Theorem 3]{Aishima2015}.

\begin{algorithm}
  \textbf{given} $A \in \mathbb{S}^n$ and an $n \times m$ thin matrix $S$ that spans the trial subspace\;
  orthonormalize $S$\;
  $(\tilde \Lambda, \tilde W) \leftarrow$ Eigendecomposition of $S^T A S$  with $\tilde \Lambda_{(1, 1)} \leq \dots \leq \tilde \Lambda_{(m, m)}$\; % \forall i = 1,\dots, m - 1$\;
  \textbf{return} the \emph{Ritz vectors} $S \tilde W$ and \emph{Ritz values} $\tilde \Lambda$ of $A$ on $\text{span}(S)$\;
  \caption{The Rayleigh-Ritz Procedure}
  \label{admm:alg:rayleigh_ritz}
\end{algorithm}
Similarly to the gradient descent method for linear systems, the iterates of \eqref{admm:eqn:gradient_ascent} lie in the Krylov subspace $\mathcal{K}_k(A, x_0)$. As a result, the following Algorithm
\begin{equation} \label{admm:eqn:lanczos_naive}
  % (\forall k \in \mathbb{N}) \quad x^{k + 1} = \argmax_{x \in %\mathcal{K}(x_0, A, k) \cap \mathbb{S}^n} {r(x)}
  \begin{array}{lll}
    (\forall k \in \mathbb{N}) \quad x^{k + 1} = &\mbox{argmax}~~r(x) \\
    & \mbox{subject to} & x \in \mathcal{K}_k(A, x_0) \\
                      && \norm{x}_2 = 1,
  \end{array}
\end{equation}
is guaranteed to yield no worse results than any variant of \eqref{admm:eqn:gradient_ascent} in finding an eigenvector associated with $\max \lambda_i$, and in practice the difference is often remarkable. But how can the Rayleigh Quotient be maximized over a subspace? This can be achieved with the \emph{Rayleigh-Ritz} Procedure, defined in Algorithm \ref{admm:alg:rayleigh_ritz}, which computes approximate eigenvalues/vectors (called \emph{Ritz values/vectors}) that are restricted to lie on a certain subspace %spanned by the columns of a given thin matrix
and are, under several notions, optimal \cite[11.4]{Parlett1998} (see discussion after Theorem \ref{admm:thm:bound}). Indeed, every iterate $x^{k + 1}$ of \eqref{admm:eqn:lanczos_naive} coincides with the last column of $X^{k +1}$, i.e. the largest Ritz vector, of the following Algorithm \cite[Theorem 11.4.1]{Parlett1998}
\begin{align} \label{admm:eqn:lanczos_rr}
  (\forall k \in \mathbb{N}) \quad
  (\Lambda^{k + 1}, X^{k + 1}) =
  &\text{ Rayleigh Ritz of } A \text{ on the trial }\\
  \nonumber
  &\text{ subspace } [x^0 \; Ax^0 \; \dots \; A^k x^0].
  % x^{k + 1} &= i\text{-th column of } \tilde V^k, \text{ where } i \text{ is the index of } \max_{i = 1, \dots k}{\lambda_i}.
\end{align}
%Note that \eqref{admm:eqn:lanczos_rr} has a complexity of $\mathcal{O}(k n^2)$, which can be computationally prohibitive.
Note that unlike \eqref{admm:eqn:lanczos_naive}, Algorithm \eqref{admm:eqn:lanczos_rr} provides approximations to not only one, but $k$ eigenpairs, with the extremum ones exhibiting a faster rate of convergence.

Remarkably, similarly to the Conjugate Gradient algorithm, \eqref{admm:eqn:lanczos_rr} and  \eqref{admm:eqn:lanczos_naive} also admit an efficient implementation, in the form of three-term recurrences known as the \emph{Lanczos Algorithm} \cite[\S 10.1]{Golub2013}. In fact, the Lanczos Algorithm produces a sequence of orthonormal vectors %$[w_1, \dots, w_k] = W$
that tridiagonalize $A$. Given this sequence of vectors, the computation of the associated Ritz pairs is inexpensive \cite[8.4]{Golub2013}. The Lanczos Algorithm is usually the method of choice for computing a few extreme eigenpairs for a symmetric matrix. However, although the Lanczos Algorithm is computationally efficient,
%, since it is based on two term recursions
the Lanczos process can suffer from lack of orthogonality, with the issue becoming particularly obvious when a Ritz pair is close to converging to some (usually extremal) eigenpair \cite{Paige1980}.
% Gradually, orthogonality of the iterates is lost
Occasional re-orthogonalizations, with a cost of $O(n^2 l^k)$ where $l^k$ is the dimension of the $k-$th trial subspace, are required to mitigate the effects of the numerical instability. To avoid such a computational cost, the Krylov subspace is restarted or shrunk so that $l^k$, and thus the computational costs of re-othogonalizations, are bounded by an acceptable amount. The Lanczos Algorithm with occasional restarts is the approach employed by the popular eigensolver \texttt{ARPACK} \cite{Lehoucq1998} for symmetric matrices.

However, there are two limitations of the Lanczos Algorithm. Namely, it does not allow for efficient warm starting of multiple eigenvectors since its starting point is a single eigenvector, and it cannot detect the multiplicity of the approximated eigenvalues as it normally provides a single approximate eigenvector for every invariant subspace of $A$.

% These issues, are particularly for using to repeatedly compute projections to the semidefinite cone for Algorithm \ref{admm:alg:admm_approximate}

\emph{Block Lanczos}
%, also known as \emph{Subspace Iteration},
addresses both of these issues. Similarly to the standard Lanczos Algorithm, Block Lanczos computes Ritz pairs on the trial block Krylov Subspace $\mathcal{K}_k(A, X_0) \eqdef \text{span}(X^0, AX^0, \dots, A^k X^0)$ where $X_0$ is an $n \times m$ matrix that contains a set of initial eigenvector guesses. Thus, Block Lanczos readily allows for the warm starting of multiple Ritz pairs. Furthermore, block methods handle clustered and multiple eigenvectors (of multiplicity up to $m$) well. However, these benefits comes at the cost of higher computational costs, as the associated subspace is increased by $m$ at every iteration. This, in turn, requires more frequent restarts, particularly for the case where $m$ is comparable to $n$.

In our experiments we observed that a single block iteration often provides Ritz pairs that give good enough projections for Algorithm \ref{admm:alg:admm_approximate}. This remarkably good performance motivated us to use the \emph{Locally Optimal Block Preconditioned Conjugate Gradient Method} (LOBPCG), presented in the following subsection.
\subsection{LOBPCG: The suggested eigensolver} \label{admm:subsec:lobpcg}
 LOBPCG \cite{Knyazev2001} is a block Krylov method that, after the first iteration, uses the trial subspace $\text{span}(X^k, AX^k, \Delta X^k)$, where $\Delta X^k \eqdef X^k - X^{k - 1}$, and sets $X^{k + 1}$ to Ritz vectors corresponding to the $m$ largest eigenvalues. Thus, the size of the trial subspace is fixed to $3m$. As a result, LOBPCG keeps its computational costs bounded and is particularly suitable for obtaining Ritz pairs of modest accuracy, as it not guaranteed to exhibit the super-linear convergence of Block Lanczos \cite{Bai2000} which might only be observed after a large number of iterations.
 Algorithm \ref{admm:alg:lobpcg} presents LOBPCG for computing the positive eigenpairs of a symmetric matrix\footnote{Note that Algorithm \ref{admm:alg:lobpcg} performs Rayleigh-Ritz on the subspace spanned by $[X^k \; AX^k -  X^k \Lambda^k\; \Delta X^k]$. Since $\Lambda^k$ is diagonal, this is mathematically the same as using $[X^k \; AX^k \; \Delta X^k]$ but using $AX^k -  X^k \Lambda^k$ improves the conditioning of the Algorithm.}. Note that the original LOBPCG Algorithm \cite[Algorithm 5.1]{Knyazev2001} is more general in the sense that it allows for the solution of generalized eigenproblems and supports preconditioning. We do not discuss these features of LOBPCG as they are not directly relevant to Algorithm \ref{admm:alg:admm_approximate}. On the other hand \cite{Knyazev2001} assumes that the number of desired eigenpairs is known a priori. However, this is not the case for $\Pi_{\mathbb{S}_+}$, where the computation of all \emph{positive} eigenpairs is required.

 In order to allow the computation of all the positive eigenpairs, $X^{k}$ is expanded when more than $m$
 % (where $m$ is the number of columns of $X^k$)
 positive eigenpairs are detected in the Rayleigh-Ritz Procedure in Line 6 of Algorithm \ref{admm:alg:lobpcg}. Note that the Rayleigh-Ritz Procedure produces $3m$ Ritz pairs, of which usually $n$ are approximate eigenpairs, (or $2m$ in the first iteration of LOBPCG) and the number of positive Ritz values is always no more than the positive eigenvalues of $A$ \cite[10.1.1]{Parlett1998}, thus the subspace $X^k$ must be expanded when more than $m$ positive Ritz values are found.

 It might appear compelling to expand the subspace to include \emph{all} the positive Ritz pairs computed by Rayleigh-Ritz. However, this can lead to ill-conditioning, as we proceed to show. Indeed, consider the case where we perform LOBPCG starting from an initial matrix $X^0$. In the first iteration, Rayleigh Ritz is performed on $\text{span}(X^0, AX^0)$. Suppose that all the Rayleigh values are positive and we thus decide to include all of the Ritz vectors in $X^1$, setting $X^1 = [X^0\; AX^0]W$ for some nonsingular $W$. In the next iteration we perform Rayleigh Ritz on the subspace spanned by
 \begin{equation*}
  \begin{bmatrix} X^1 & AX^1 & \Delta X^1 \end{bmatrix} =
    \begin{bmatrix}
      X^0 & AX^0 & AX^0 & A^2X^0 & \Delta X^1
    \end{bmatrix}
    \begin{bmatrix}
      W & & \\
      & W & \\
      & & I
    \end{bmatrix}.
  \end{equation*}
The problem is that the above matrix is rank deficient.  Thus one has to rely on a numerically stable Algorithm, like Householder QR, for its orthonormalization (required by the Rayleigh-Ritz Procedure) instead of the more efficient Cholesky QR algorithm \cite[page 251]{Stewart1998}.
% that could be used for well conditioned matrices.
Although, for this example, one can easily reduce columns from the matrix so that it becomes full column rank, the situation becomes more complicated when not all of the Rayleigh values are positive.
% or when the subspace is expanded in subsequent iterations when $\Delta X^k$ is also present.
In order to avoid this numerical instability, and thus be able to use Cholesky QR for othonormalizations, we expand $X^k$ whenever necessary by a fixed size (equal to a small percentage of $n$ (the size of $A$), e.g. $n/50$) with a set of randomly generated vectors.
\begin{algorithm}
  \textbf{given} $A \in \mathbb{S}^n$ and the $n \times m$ thin matrix $X^0$ that spans the initial trial subspace\;
  $(\Lambda^0, X^0) \leftarrow$ Rayleigh-Ritz for $A$ on the trial subspace $\text{span}(X^0)$\;
  $\Delta X^0 \leftarrow$ empty $n \times 0$ matrix\;
  \For{$k = 0, \dots$ until convergence}{
    $R^k \leftarrow A X^k - X^k \Lambda^k$\;
    $(\Lambda^{k + 1}, X^{k + 1}) \leftarrow$ Apply Rayleigh-Ritz
    for $A$ on the trial subspace $\text{span}(X^k, R^k, \Delta X^k)$ and return the $m$ largest eigenpairs\;
    $\Delta X^{k + 1} \leftarrow X^{k + 1} - X^k$\;
    Expand $\Lambda^{k + 1}, X^{k + 1}$ with randomly generated elements and set $m = \text{size}(X^{k + 1}, 2) = \text{size}(\Lambda^{k + 1}, 2)$ if the positive Ritz values of line 6 were more than $m$.
  }
  \textbf{return} $X^k, \Lambda^{k}$ containing $m$ Ritz pairs that approximate the positive eigenpairs of $A$
  \caption{The LOBPCG Algorithm for Computing the Positive Eigenpairs of a Symmetric Matrix}
  \label{admm:alg:lobpcg}
\end{algorithm}
\paragraph{When is projecting to $\mathbb{S}_+$ with LOBPCG most efficient?}
Recall that there exist two ways to project a matrix $A$ into the semidefinite cone. The first is to compute all the positive eigenpairs $\Lambda_+, V_+$ of $A$ and set $\Pi_{\mathbb{S}_+}(A) = V_+ \Lambda_+ V_+^T$.
%This is approach is suitable for when $A$ has mostly negative eigenpairs.
The opposite approach is to compute all the \emph{negative} eigenpairs $\Lambda_-, V_-$ of $A$ and set $\Pi_{\mathbb{S}_+}(A) = I - V_- \Lambda_- V_-^T$.
%which is preferable when $A$ has mostly positive eigenpairs.
The per-iteration cost of LOBPCG is $O(n^2 m)$ where $m$ is the number of computed eigenpairs. Thus, when most of the eigenvalues are nonpositive, then the positive eigenpairs should be approximated, and vice versa.

As a result, LOBPCG is most efficient when the eigenvalues of the matrix under projection are either almost all nonnegative or almost all nonpositive, in which case LOBPCG exhibits an almost quadratic complexity, instead of the cubic complexity of the full eigendecomposition. This is the case when ADMM converges to a low rank primal or dual solution of \eqref{admm:eqn:main_problem}. Fortunately, low rank solutions are often present or desirable in practical problems \cite{Lemon2016}. On the other hand, the worst case scenario is when half of the eigenpairs are nonpositive and half nonnegative, in which case LOBPCG exhibits worse complexity than the full eigendecomposition and thus the latter should be preferred.

\subsection{Error Analysis \& Stopping Criteria} \label{admm:subsec:termination}
Algorithm \ref{admm:alg:admm_approximate} requires that the approximation errors in lines 3 and 5 are bounded by a summable sequence. As a result, bounds on the accuracy of the computed solutions are necessary to assess when the approximate algorithms (CG and LOBPCG) can be stopped.

For the approximate solution of the Linear System \eqref{admm:eqn:kkt_system} one can easily devise such bounds. Indeed, note that the left hand matrix of \eqref{admm:eqn:kkt_system} is fixed across iterations and is full rank.
We can check if an approximate solution $[\bar x^{k + 1}; \; \newcontent{\bar z^{k + 1}}]$ satisfies the condition
\begin{equation} \label{admm:eqn:kkt_termination}
  \norm{
\begin{bmatrix}
  \tilde x^{k+1} \\
  \tilde z^{k + 1}
\end{bmatrix}
-
\begin{bmatrix}
  \bar x^{k + 1} \\
  \newcontent{\bar z^{k + 1}}
\end{bmatrix}
  }_2 \leq \mu^k
\end{equation}
of Algorithm \ref{admm:alg:admm_approximate} easily, since (recalling Q is the KKT matrix defined in \ref{admm:eqn:kkt_system})
\begin{equation}
  \norm{
\begin{bmatrix}
  \tilde x^{k+1} \\
  \tilde z^{k + 1}
\end{bmatrix}
-
\begin{bmatrix}
  \bar x^{k + 1} \\
  \newcontent{\bar z^{k + 1}}
\end{bmatrix}
  }_2
\leq
  \norm{Q^{-1}}
  \underbrace{
  \norm{
  b^k
  -
Q
\begin{bmatrix}
  \bar x^{k + 1} \\
  \newcontent{\bar z^{k + 1}}
\end{bmatrix}
  }_2}_{\eqdef r^k}.
\end{equation}
Since $\norm{Q^{-1}}$ is constant across iterations, it can be ignored when considering the summability of the approximation errors \ref{admm:eqn:kkt_termination}.
Thus, we can terminate CG (or any other iterative linear system solver employed) when the \emph{residual} $r^k$ of the approximate solution $[\bar x^{k + 1};\; \newcontent{\bar z^{k + 1}}]$ becomes less than a summable sequence e.g. $1/k^2$.

On the other hand, controlling the accuracy of the projection to the Semidefinite Cone requires a closer examination. Recall that, given a symmetric matrix $A$ that is to be projected\footnote{Note that the matrices under projection depend on the iteration number of ADMM. We do not make this dependence explicit in order to keep the notation uncluttered.}, our approach uses LOBCPG to compute a set of positive Ritz pairs $\tilde V, \tilde \Lambda$ approximating $V_+, \Lambda_+$ of \eqref{admm:eqn:eigendecomposition} which we then use to approximate  $\Pi_{\mathbb{S}_+^n}(A) = V_+ \Lambda_+ V_+^T$ as\footnote{When LOBPCG approximates the negative eigenspace (because the matrix under projection is believed to be almost positive definite), then all of the results of this section hold mutatis mutandis. Refer to for more details.} $\tilde V \tilde \Lambda \tilde V^T$.
% In this subsection we address the question ``How accurate is this approximation?''
A straightforward approach would be to quantify the projection's accuracy with respect to the accuracy of the Ritz pairs. Indeed, if we assume that our approximate positive eigenspace is ``sufficiently rich'' in the sense that $\lambda_{\max}(\tilde V_{\perp} A \tilde V_{\perp}) \leq 0$, then we get $m = \tilde m$ \cite[Theorem 10.1.1]{Parlett1998}, thus we can define $\Delta \Lambda = \Lambda_+ - \tilde \Lambda$, $\Delta V = V_+ - \tilde V$ which then gives the following bound
\begin{equation}
  \norm{V_+ \Lambda_+ V_+^T - \tilde V \tilde \Lambda \tilde V^T} \leq
  2 \norm{\Delta V \Lambda_+ V_+^T} + \norm{V_+ \Delta \Lambda V_+^T} + O(\norm{\Delta}^2)
\end{equation}
with $\norm{\Delta} \eqdef \max(\norm{\Delta V}, \norm{\Delta \Lambda})$. Standard results of eigenvalue perturbation theory can be used to bound the error in the computation of the eigenvalues, i.e. by $\norm{\Delta \Lambda}^2_F \leq 2 \norm{R}^2_F$ \cite[Theorem 11.5.2]{Parlett1998}\footnote{Note that following \cite[Theorem 11.5.1]{Parlett1998}, $\lambda_{\max}(\tilde V_{\perp} A \tilde V_{\perp}) \leq 0$ implies that the indices of $\alpha$ can coincide with the indices of $\theta$ in \cite[Theorem 11.5.2]{Parlett1998}.} where
$$
R \eqdef A \tilde V - \tilde V \newcontent{\tilde \Lambda}.
$$
In contrast, $\norm{\Delta V}$ is ill-conditioned, as the eigenvectors are not uniquely defined in the presence of multiple (i.e. clustered) eigenvalues. At best, eigenvalue perturbation theory can give $\norm{\Delta V} \lessapprox \norm{R}/\text{gap}$ \cite[Theorem 3.1 and Remark 3.1]{Nakatsukasa2020} where
$$\text{gap} \eqdef \min_{i, j}\left(\tilde \Lambda_{(i, i)} - \Lambda_{-{(j, j)}}\right).$$ This implies that the projection accuracy depends on the separation of the spectrum and can be very poor in the presence of small eigenvalues. Note that unlike $R$ that is readily computable from $(\tilde V, \tilde \Lambda)$, $\text{``gap''}$ is, in general, unknown and non-trivial to compute, thus further complicating the analysis.

To overcome these issues, we employ a novel bound that shows that, although the accuracy of the Ritz pairs depends on the separation of eigenvalues, the approximate projection does not:
\begin{theorem} \label{admm:thm:bound}
  Assume that $\tilde V$ and $\tilde \Lambda$ are such that $\tilde \Lambda = \tilde V ^T A \tilde V$. Then
  $$
    \norm{\tilde V \tilde \Lambda \tilde V^T - \Pi_{\mathbb{S}_+}(A)}_F^2 \leq 2 \norm{R}^2_F + \norm{\Pi_{\mathbb{S}_+}(\tilde V_\perp^T A \tilde V_\perp)}_F^2
  $$
\end{theorem}
\begin{proof}
  This is a restatement of \cite[Corollary 2.1]{Goulart2020}.
\end{proof}
Note that the above result does not depend on the assumption that $\lambda_{\max}(\tilde V_{\perp} A \tilde V_{\perp})$ is nonpositive or that $m = \tilde m$.
Nevertheless, with a block Krylov subspace method it is often expected that $\lambda_{\max}(\tilde V_{\perp} A \tilde V_{\perp})$ will be either small or negative, thus the bound of Theorem \ref{admm:thm:bound} will be dominated by $\norm{R}$. The assumption $\tilde \Lambda = \tilde V ^T A \tilde V$ is satisfied when $\tilde V$ and $\tilde \Lambda$ are generated with the Rayleigh Ritz Procedure and thus holds for Algorithm \ref{admm:alg:lobpcg}. In fact, the use of the Rayleigh-Ritz, which is employed by Algorithm \ref{admm:alg:lobpcg}, is strongly suggested by Theorem \ref{admm:thm:bound} as it minimizes $\norm{R}_F$ \cite[Theorem 11.4.2]{Parlett1998}.

We suggest terminating Algorithm \ref{admm:alg:lobpcg} when every positive Ritz pair has a residual with norm bounded by a sequence that is summable across ADMM's iterations. Then, excluding the effect of $\norm{\Pi_{\mathbb{S}_+}(\tilde V_\perp^T A \tilde V_\perp)}_F^2$, which appears to be negligible according to the results of the next section, Theorem \ref{admm:thm:bound} implies that  the summability requirements of Algorithm \ref{admm:alg:lobpcg} will be satisfied.
Either way, the term $\norm{\Pi_{\mathbb{S}_+}(\tilde V_\perp^T A \tilde V_\perp)}_F^2$ can be bounded by $(n-\tilde m)\lambda_{\max}^2(\tilde V_{\perp} A \tilde V_{\perp})$, where $\lambda_{\max}^2(\tilde V_{\perp} A \tilde V_{\perp})$ can be estimated with a projected Lanczos methods.

% As we will see in Section \ref{admm:sec:results}, Algorithm \ref{admm:alg:lobpcg} is very efficient in ADMM and quite often one iteration produces sufficiently good projections. As a result, instead of the $O(n^3)$ that the exact eigendecomposition would require projecting with \ref{admm:alg:lobpcg} has a complexity of $O(n^2 m)$ where $m$ are the number of positive eigenvalues in the matrix under projection.

% \subsection{Practical Considerations}
% Even if we ignore practical issues of eigensolvers such as \texttt{ARPACK} (see Section [REF]), standard eigenvalue perturbation cannot warranty this. Indeed, suppose an approximate eigenpair $(\hat \lambda, \hat \phi)$ with an associated residual $r \eqdef A \hat \phi - \lambda \phi$. Although eigenvalues are well conditioned for symmetric matrices \cite{Parlett1998} and their accuracy can be bounded e.g. with the Weyl Bounds
% \begin{equation}
%   \min_{j}\abs{\hat \lambda - \lambda_j} \leq \norm{r}
% \end{equation}
% the computation of independent eigenvectors $\phi$ is ill conditioned, in the sense that it depends on the spectral gap \cite{Davis1970}
% \begin{equation}
%   \text{gap}_i \eqdef \min_{j \neq i}\abs{\lambda_j - \hat \lambda}
% \end{equation}

% \begin{equation}
%   \sin \angle(v_i, \hat v) = \norm{r}/\text{gap}_i
% \end{equation}

\section{Experiments and Software}
In this section we provide numerical results for Semidefinite Programming with Algorithm \ref{admm:alg:admm_approximate}, where the projection to the Semidefinite Cone is performed with Algorithm \ref{admm:alg:lobpcg}. Our implementation is essentially a modification of the optimization solver \texttt{COSMO.jl}. \texttt{COSMO.jl} is an open-source Julia implementation of Algorithm \ref{admm:alg:admm_approximate} which allows for the solution of problems in the form \eqref{admm:eqn:main_problem} for which $\mathcal{C}$ is a composition of translated cones $\{\mathcal{K}_i + b_i\}$. Normally, \texttt{COSMO.jl} computes ADMM's steps to machine precision and supports any cone $\mathcal{K}_i$ for which a method to calculate its projection is provided\footnote{Operations for testing if a vector belongs to $\mathcal{K}_i$, its polar and its recession must be provided. These operations might be used to check for termination of the Algorithm, which, by default, is checked every 40 iterations. For the Semidefinite Cone, both of these tests can be implemented via the Cholesky factorization.}. \texttt{COSMO.jl} provides default implementations for various cones, including the Semidefinite cone, where \texttt{LAPACK}'s \texttt{syevr} function is used for its projection. The modified solver used in the experiments of this section can be found online:
\centerline{\texttt{https://github.com/nrontsis/ApproximateCOSMO.jl}}.
Code reproducing the results of this section is also publicly available.\footnote{For subsections \ref{admm:subsec:sdplib} and \ref{admm:subsec:infeasible} at \texttt{https://github.com/nrontsis/SDPExamples.jl}.}

We compared the default version of \texttt{COSMO.jl} with a version where the operation \texttt{syevr} for the Semidefinite Cone is replaced with Algorithm \ref{admm:alg:lobpcg}. We have reimplemented \texttt{BLOPEX}, the original  \texttt{MATLAB} implementation of LOBPCG \cite{Knyazev2001}, in Julia. For the purposes of simplicity, our implementation supports only symmetric standard eigenproblems without preconditioning. For these problems, our implementation was tested against \texttt{BLOPEX} to assure that exactly the same results (up to machine precision) are returned for identical problems. Furthermore, according to \S\ref{admm:subsec:lobpcg}
%and unlike \texttt{BLOPEX}
we provide the option to compute all eigenvalues that are larger or smaller than a given bound. %Our Julia implementation of LOBCPG can be found online at:\\
%\centerline{\texttt{https://github.com/nrontsis/LOBPCG.jl}}

At every iteration $k$ of Algorithm \ref{admm:alg:admm_approximate} we compute approximate eigenpairs of every matrix that is to be projected onto the semidefinite cone. If, at the previous iteration of ADMM, a given matrix \newcontent{was} estimated to have less than a third of its eigenvectors positive, then LOBPCG is used to compute its \emph{positive} eigenpairs, according to \eqref{admm:eqn:projection} (middle). If it had less than a third of its eigenvectors negative, then LOBPCG computes its negative eigenpairs according to \eqref{admm:eqn:projection} (right). Otherwise, a full eigendecomposition is used.

In every case, LOBPCG is terminated when all of the Ritz pairs have a residual with norm less than $10/k^{1.01}$.
% This  Assuming that $\delta_k \eqdef \norm{\Pi_{\mathbb{S}_+}(\tilde {V^k_\perp}^T A^k \tilde V^k _\perp)}_F^2$ is zero (or very small compared to $\norm{R^k}^2_F$), Theorem \ref{admm:thm:bound} gives
% \begin{equation}
%  \norm{\tilde V^k \tilde \Lambda^k
%  \tilde {V^k}^T
%  - \Pi_{\mathbb{S}_+}(A^k)}_F^2 \leq 2 \norm{R^k}^2_F
%\end{equation}
According to \S\ref{admm:subsec:termination}, this implies that the projection errors are summable across ADMM's iterations, assuming that the rightmost term of Theorem \ref{admm:thm:bound} is negligible. Indeed, in our experiments, these terms were found to converge to zero very quickly, and we therefore ignored them. A more theoretically rigorous approach would require the consideration of these terms, a bound of which can obtained using e.g. a projected Lanczos algorithm, as discussed in \S\ref{admm:subsec:termination}.

The linear systems of Algorithm \ref{admm:alg:admm_approximate} are solved to machine precision via an LDL factorization \cite[\S 16.2]{Nocedal2006}. We did not rely on an approximate method for the solution of the linear system because, in the problems that we considered, the projection to the Semidefinite Cone required the majority of the total time of Algorithm \ref{admm:alg:admm_approximate}. Nevertheless, the analysis of presented in Sections \ref{admm:sec:algorithm}-\ref{admm:sec:proof} allows for the presence of approximation errors in the solution of the linear systems.

\subsection{Results for the SDPLIB collection} \label{admm:subsec:sdplib}
We first consider problems of the SDPLIB collection, in their dual form, i.e.
\begin{equation}\label{admm:eqn:sdplib}
  \begin{array}{ll}
    \mbox{maximize}   & \langle F_0, Y\rangle \\
    \mbox{subject to} & \langle F_i, Y \rangle = c_i, \quad Y \in \mathbb{S}^n_+.
  \end{array}
\end{equation}
The problems are stored in the sparse SDPA form, which was designed to efficiently represent SDP problems in which the matrices $F_i, i = 0, \dots m$ are block diagonal with sparse blocks. If the matrices $F_i$ consist of $\ell$ diagonal blocks, then the solution of \eqref{admm:eqn:sdplib} can be obtained by solving
\begin{equation}\label{admm:eqn:sdplib_split}
  \begin{array}{ll}
    \mbox{maximize}   & \sum_{j = 1}^\ell \langle F_{0, j}, Y_j \rangle \\
    \mbox{subject to} & \sum_{j = 1}^\ell \langle F_{i, j}, Y_j \rangle = c_i, \quad  Y_j \in \mathbb{S}^{n_j}_+ \quad j = 1, \dots, \ell.
  \end{array}
\end{equation}
where $F_{i, j}$ denotes the $j-$th diagonal block of $F_i$ and $Y_j$ the respective block of $Y$. Note that \eqref{admm:eqn:sdplib_split} has more but smaller semidefinite variables than \eqref{admm:eqn:sdplib}; thus it is typically solved by solvers like COSMO.jl more efficiently than \eqref{admm:eqn:sdplib}. As a result, our results refer to the solution of problems in the form \eqref{admm:eqn:sdplib_split}.

Table \ref{admm:tab:sdplib}, presented in the Appendix,
shows the results on all the problems of SDPLIB problems for which the largest semidefinite variable is of size at least $50$. We observe that our approach can lead to a significant speedup of up to 20x. At the same time, the robustness of the solver is not affected, in the sense that the number of iterations to reach convergence is not, on average, increased by using approximate projections. It is remarkable that for every problem that the original \texttt{COSMO.jl} implementation converges within 2500 iterations (i.e. the default maximum iteration limit), our approach also converges with a faster overall solution time.
\subsection{Infeasible Problems} \label{admm:subsec:infeasible}
Next, we demonstrate the asymptotic behavior of Algorithm \ref{admm:alg:admm_approximate} on the problem \texttt{infd1} of the SDPLIB collection. This problem can be expressed in the form \eqref{admm:eqn:main_problem} with  $\mathcal{C} = \set{\vc_u(X)}{X \in \mathbb{S}^{30}}$ (the set of vectorized $30 \times 30$ positive semidefinite matrices), and $x \in \mathbb{R}^{10}$.

As the name suggests, \texttt{infd1} is dual infeasible. Following \cite[\S 5.2]{osqpinfeasibility}, \texttt{COSMO} detects dual infeasibility in conic problems when the certificate \eqref{admm:eqn:dual_infeasibility} holds approximately, that is when $\delta x^k \neq 0$ and
\begin{equation*}
  \text{dist}_{{C}^{\infty}}\left(A \bar x^k \right) < \epsilon_\text{dinf}, \quad \text{and} \quad q^T \bar x^k < \epsilon_\text{dinf},
\end{equation*}
where $\bar x^k \eqdef \delta x^k/ ||\delta x^k||$, for a positive tolerance $\epsilon_\text{dinf}$. Figure \ref{admm:fig:dual_infeasibility}, depicts the convergence of these quantities both for the case where the projection to the semidefinite cone are computed approximately and when LOBPCG is used. The convergence of the successive differences to a certificate of dual infeasibility is practically identical.
\begin{figure}
	\centering
	\begin{minipage}{.49\textwidth}
	\includegraphics[width=.99\textwidth]{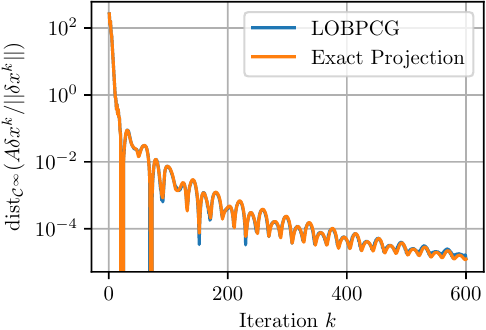}
	\end{minipage}
	\centering
	\begin{minipage}{.49\textwidth}
	\includegraphics[width=.99\textwidth]{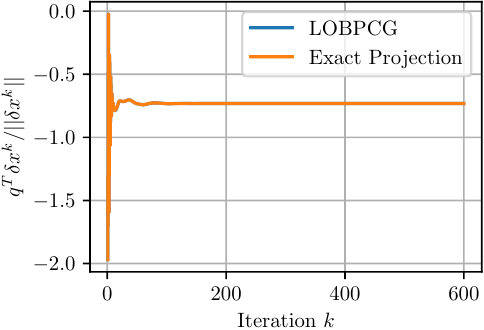}
	\end{minipage}
  \caption{Convergence of $\left(\delta x^k/\norm{\delta x^k} \right)_{k \in \mathbb{N}}$ to a certificate of dual infeasibility for problem \texttt{infd1} from the SDPLIB. A fixed value of $\rho = 10^{-3}$ is used in Algorithm \ref{admm:alg:admm_approximate}.}
  \label{admm:fig:dual_infeasibility}
\end{figure}

To demonstrate the detection of primal infeasibility we consider the dual of \texttt{infd1}. Following \cite[\S 5.2]{osqpinfeasibility},
\texttt{COSMO} detects primal infeasibility in conic problems when the certificate \eqref{admm:eqn:primal_infeasibility} is satisfied approximately, that is when $\delta y^k \neq 0$ and
\begin{equation*}
  \norm{P \bar y^k} < \epsilon_\text{pinf}, \quad
  \norm{A^T \bar y^k} < \epsilon_\text{pinf}, \quad \text{dist}_{{C}^{\circ}}\left(\bar y^k \right) < \epsilon_\text{pinf}, \quad b^T \bar y^k < \epsilon_\text{pinf},
\end{equation*}
where $\bar y^k \eqdef \delta y^k / || \delta y^k||$, for a positive tolerance $\epsilon_\text{pinf}$. Note that, for the case of the dual of \texttt{infd1}, the first condition is trivial since $P = 0$. Figure \ref{admm:fig:primal_infeasibility} compares the convergence of our approach, against standard COSMO, to a certificate of infeasibility. LOBPCG yields practically identical convergence as the exact projection for all of the quantities except $\norm{A^T \bar y^k}$, where slower convergence is observed.
\begin{figure}
	\centering
	\begin{minipage}{.49\textwidth}
	\includegraphics[width=.99\textwidth]{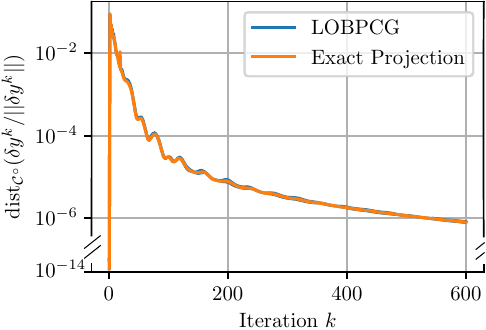}
	\end{minipage}
	\centering
	\begin{minipage}{.49\textwidth}
	\includegraphics[width=.99\textwidth]{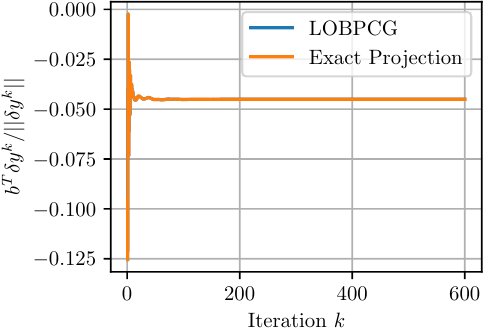}
  \end{minipage}
  \begin{minipage}{.49\textwidth}
    \includegraphics[width=.99\textwidth]{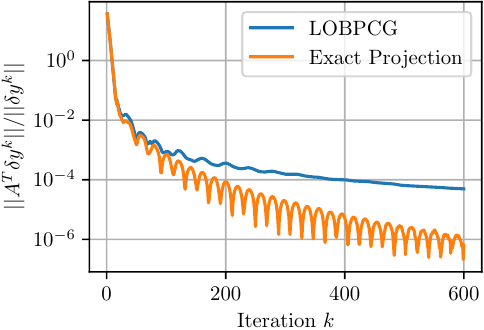}
    \end{minipage}
  \caption{Convergence of $\left(\delta y^k / \norm{\delta y^k} \right)_{k \in \mathbb{N}}$ to a certificate of primal infeasibility for the dual of the problem \texttt{infd1} from the SDPLIB. A fixed value of $\rho = 10^{3}$ is used in Algorithm \ref{admm:alg:admm_approximate}.}
  \label{admm:fig:primal_infeasibility}
\end{figure}

Note that SDPLIB also contains two instances of primal infeasible problems: \texttt{infp1} and \texttt{infp2}. However, in these problems, there is a single positive semidefinite variable of size $30 \time 30$ and, in ADMM, the matrices projected to the semidefinite cone have rank $15 = 30/2$ across all the iterations (except for the very first few). Thus, according to \S\ref{admm:subsec:lobpcg} LOBPCG yields identical results to the exact projection, hence a comparison would be of little value.

\section{Conclusions}
We have shown that state-of-the art approximate eigensolvers can bring significant speedups to ADMM for the case of Semidefinite Programming. We have extended the results of \cite{osqpinfeasibility} to show that infeasibility can be detected even in the presence of appropriately controlled projection errors, thus ensuring the same overall asymptotic behavior as an exact ADMM method. Future research directions include exploring the performance of other state-of-the-art eigensolvers from the Linear Algebra community \cite{Stathopoulos2010}.
\section*{Acknowledgements} This work was supported by the EPSRC AIMS CDT grant EP/L015987/1 and Schlumberger. We would also like to acknowledge the support of the NII International Internship Program, which funded the first author's visit to NII, during which this collaboration was launched.

\appendix
\section{Convergence of approximate iterations of nonexpansive operators} \label{admm:app:delta}
In this section we provide a proof for Theorem \ref{admm:thm:delta_approx}. We achieve this by generalizing some of the results of \cite{Pazy1971}, \cite{Bailion1978} \newcontent{and \cite{Ishikawa1976}} to account for sequences generated by approximate evaluation of \newcontent{averaged} operators $T$ for which $\cl{\mathcal{R}(\Id - T)}$ has the \emph{minimum property} defined below:
\begin{definition}[Minimum Property]
  Let $K \subseteq \mathcal{H}$ be closed and let $\ell$ be the \newcontent{minimum-norm} element of $\cl{\conv K}$. The set $K$ has the minimum property if $\ell \in K$.
\end{definition}
Note that $\cl{\mathcal{R}(\Id - T)}$ has the minimum property when $T$ is defined as \eqref{admm:eqn:admm_operator} because the domain of \eqref{admm:eqn:admm_operator} is convex \cite[Lemma 5]{Pazy1971}. Thus Theorem \ref{admm:thm:delta_approx} follows from the following result:
\begin{proposition} \label{admm:prop:delta_approx}
  Consider some $\mathcal{D} \subseteq \mathcal{H}$ that is closed, an \newcontent{averaged} $T : \mathcal{D} \rightarrow \mathcal{D}$ and assume that $\cl{\mathcal{R}(\Id - T)}$ has the minimum property. For any sequence defined as 
  $$(\forall k \in \mathbb{N}) \quad x^{k + 1} \approx_{\epsilon^k} T x^k,$$
  for some $x^0 \in \mathcal{D}$ and a summable nonnegative sequence $(\epsilon^k)_{k \in \mathbb{N}}$, we have $$\lim_{k \to \infty}(x^{k + 1} - x^k) = \lim_{k \to \infty} {x^k}/{k} = -\ell,$$ where $\ell$ is the unique element of minimum norm in $\cl{\mathcal{R}(\Id - T)}$.
\end{proposition}

\newcontent{
To prove Proposition \ref{admm:prop:delta_approx} we will need the following Lemma:
\begin{lemma} \label{admm:lem:ishikawa}
For any approximate iteration over an averaged operator, defined as:
$$x^{k + 1} = (1 - t) x^k + t \tilde T x^k + \delta^k,$$
where $x_0 \in \mathcal{D} \subseteq \mathcal{H}$, $\tilde T : \mathcal{D} \mapsto \mathcal{D}$ is a nonexpansive operator, $t \in [0, 1)$ and $(\norm{\delta^k})_{k \in \mathbb{N}}$ is some real, summable sequence representing approximation errors in the iteration, we have:
\begin{equation} \label{admm:eqn:main}
  \frac{1}{N}\lim_{k \to \infty}\norm{x^{k + N}- x^{k}} = \lim_{k \to \infty} \norm{x^{k+1} - x^k}.
\end{equation}
\end{lemma}
\begin{proof}
When $(\norm{\delta^k})_{k \in \mathbb{N}}$ is zero, the proof for \eqref{admm:eqn:main} is mentioned as ``a straightforward modification of Ishikawa's argument'' in \cite[Proof of Theorem 2.1]{Bailion1978}; we show it in detail and for any summable $(\norm{\delta^k})$, indeed by modifying \cite[Proof of Lemma 2]{Ishikawa1976}.

We will first show that $\lim_{k \to \infty} \norm{x^{k + 1} - x^k}$ and $\lim_{k \to \infty} \norm{x^{k + N} - x^k}$ exist and are bounded. To this end, consider the sequence
$$
q^k \eqdef \norm{
  x^{k + 1} -
  x^k} + \sum_{i = k}^{\infty}\norm{\delta^{i} - \delta^{i - 1}}.
$$
Note that
$\sum_{i = k}^{\infty}\norm{\delta^{i} - \delta^{i - 1}}$
converges to a finite value for every $k$, as it is the limit $n \to \infty$ of the nondecreasing sequence ($\sum_{i = k}^{n}\norm{\delta^{i} - \delta^{i - 1}}$) that is bounded above because $(\norm{\delta^i})$ is summable. Since $\norm{x^k - x^{k - 1}} \leq \norm{x^{k + 1} - x^k}  + \norm{\delta^{k} - \delta^{k - 1}}$ we conclude that $(q^k)$ is nonincreasing. Since $(q^k)$ is also bounded below by zero we conclude that $\lim_{k \to \infty}q^k$ exists and is bounded. Finally, because
$\lim_{k \to \infty}\sum_{i = k}^{\infty}\norm{\delta^{i} - \delta^{i - 1}} = 0,$
we conclude that $\lim_{k \to \infty}\norm{x^{k + 1} - x^k}$ also exists and is bounded.  Using similar arguments, we can show the same for $\lim_{k \to \infty} \norm{x^{k + N} - x^k}$.

Since $x^{k + 1} - x^{k} = t(\tilde Tx^k - x^k) + \delta^k$ $\forall k \in \mathbb{N}$, it suffices to show 
\begin{equation} \label{eqn:equivalent}
  \frac{1}{N}\lim_{k \to \infty}\norm{x^{k + N}- x^{k}} = t{\lim_{k \to \infty}\norm{x^k - \tilde Tx^k}}
\end{equation}
instead of \eqref{admm:eqn:main}, because $\norm{\delta^k} \to 0$. Note that both limits in the above equation exist and are bounded, as discussed above.

We will show \eqref{eqn:equivalent} by showing that 
\begin{equation*}
  t r \leq
  \frac{1}{N}\lim_{k \to \infty}\norm{x^{k + N + 1}- x^{k + 1}} \leq t r.
\end{equation*}
where $r \eqdef \norm{x^k - \tilde Tx^k}$.

The proof for both of these bounds depends on the following equality:
\begin{align} \label{eqn:sum}
  \begin{split}
  x^{k + N + 1} - x^{k + 1} &= \sum_{i=1}^N [(1-t) x^{k + i} + t \tilde T x^{k + i} + \delta^{k + i}] - x^{k + i }\\ &= \sum_{i=1}^N t (\tilde T x^{k+i} - x^{k+i}) + \delta^{k + i}.
  \end{split}
\end{align}

The upper bound follows easily from the triangular inequality:
\begin{align}
  \norm{x^{k + N + 1} - x^{k + 1}} &\leq \sum_{i=1}^N \left( t \norm{\tilde T x^{k+i} - x^{k+i}} + \norm{\delta^{k + i}} \right) \Rightarrow \\
  \label{upper_bound}
  \lim_{k \to \infty} \frac{1}{N} \norm{x^{k + N + 1} - x^{k + 1}} &\leq t \lim_{k \to \infty} \norm{x^k - \tilde Tx^k},
\end{align}
since $\norm{\delta^{k + i}} \to 0$.

To get the lower bound, define $u^i \eqdef \tilde Tx^i - x^i$ and $s \eqdef 1 - t$, and note that for all $i \in \mathbb{N}$: \begin{align}
  \label{eqn:weird_difference}
  \begin{split}
    &\norm{u^{i + 1} - s u^{i}} = \norm{\tilde T x^{i + 1} - x^{i + 1} - (1 - t)(\tilde T x^i - x^i)}\\
    &\leq \norm{\tilde T x^{i + 1} - ((1-t) x^i + t \tilde T x^i)) - (1 - t)(\tilde T x^i - x^i)} + \norm{\delta_i} \\
    &= \norm{\tilde T ((1-t) x^i + t \tilde T x^i + \delta^i) - \tilde T x^i} + \norm{\delta^i} \\
    &\leq \norm{((1-t) x^i + t \tilde T x^i ) - x^i} + 2\norm{\delta^i}\\
    &= t\norm{\tilde Tx^i - x^i} + 2\norm{\delta^i} \\
    &= (1-s) \norm{u^i} + 2\norm{\delta^i}.
  \end{split}
\end{align}
Thus, using \eqref{eqn:sum}, and since $\norm{\delta^{k +i}} \to 0$ we have
\begin{align*}
  &\lim_{k \to \infty} s^{N - 1} \norm{x^{k + N + 1} - x^{k + 1}} = \lim_{k \to \infty }\norm{s^{N-1} \sum_{i=1}^N (1-s)u^{k + i}} \\
  \geq &\limsup_{k \to \infty} \left[(1 - s^N) \norm{u^{k + N}} - \sum_{i=1}^{N-1} s^{N-1-i} (1 - s^{i}) \norm{u^{k + i + 1} - s u^{k + i}} \right], \\
  \intertext{where \cite[Lemma 1]{Ishikawa1976} was used above,}
  \geq &(1 - s^N) \liminf_{k \to \infty}\norm{u^{k+N}} - \sum_{i=1}^{N-1} s^{N-1-i} (1 - s^{i}) \liminf_{k \to \infty} \norm{u^{k + i + 1} - s u^{k + i}} \\
  \geq &(1 - s^N) r - \sum_{i=1}^{N-1} s^{N-1-i} (1 - s^{i}) (1-s) r, \\
  \intertext{because of \eqref{eqn:weird_difference} and because $2\norm{\delta^i} \to 0$,}
  = & \left[1 - s^N  - \sum_{i=1}^{N-1} s^{N-1-i} (1 - s^{i}) (1-s)\right] r 
  = r s^{N - 1} \sum_{i=1}^{N} (1 - s).
\end{align*}
where \cite[Lemma 1]{Ishikawa1976} was used in the last equality above.
Thus $$\lim_{k \to \infty} s^{N-1} \norm{x^{k + N + 1} - x^{k + 1}} \geq r s^{N - 1} \sum_{i=1}^{N} (1 - s)$$
or, since $s > 0$, we get the desired lower bound that concludes the proof
\begin{equation*}
\frac{1}{N}\lim_{k \to \infty}\norm{x^{k + N + 1}- x^{k + 1}} \geq t \lim_{k \to \infty} \norm{x^k - \tilde Tx^k}.
\end{equation*}
\end{proof}

We can now proceed with the Proof of Proposition \ref{admm:prop:delta_approx}.}
We first show $\lim_{k \to \infty} {x^k}/{k} = -\ell$. The nonexpansiveness of $T$ gives
\begin{align*}
  \norm{x^n - T^nx^0} &\leq \norm{Tx^{n - 1} - TT^{n-1}x^0} + \epsilon_n, \quad \forall n \in \mathbb{N} \\
  \Rightarrow \frac{1}{n}\norm{x^n - T^nx^0} &\leq \frac{1}{n}\sum_{i = 1}^{n}\epsilon_i \Rightarrow \lim_{n \to \infty}\norm{\frac{x^n}{n} - \frac{T^nx^0}{n}} = 0,
\end{align*}
where the summability of $(\epsilon_i)_{i \in \mathbb{N}}$ was used in the last implication. Thus, the claim follows from \cite[Theorem 2]{Pazy1971}.

It remains to show that $\lim_{k \to \infty}x^{k + 1} - x^k$ also converges to $-\ell$.
\newcontent{To this end, note that due to Lemma \ref{admm:lem:ishikawa}, we have $\lim_{k \to \infty}\norm{x^{k + N} - x^k}/N = \lim_{k \to \infty}\norm{x^{k + 1} - x^k}$.
Furthermore the non-expansiveness of $T$ gives for all $N \geq 1$:
\begin{align*}
 \lim_{k \to \infty}\norm{x^{k + N} - x^k}/N \leq \frac{1}{N}\left(\norm{x^{N} - x^0} + 2\sum_{i=0}^{\infty}\epsilon^{i}\right).
\end{align*}
Noting also that $\displaystyle \lim_{N \to \infty}\norm{x^N - x^0}/N$ exists due to the first part of the proof, we get:
\begin{align}
  \begin{split}
\lim_{k \to \infty} \norm{x^{k + 1} - x^k} &= \lim_{k \to \infty}\norm{x^{k + N} - x^k}/N \; \; \forall N \geq 1 \\
&\leq \lim_{N \to \infty}\frac{1}{N}\norm{x^N - x^0} \leq \limsup_{N \to \infty} \frac{1}{N}\sum_{i=0}^{N-1} \norm{x^{i + 1} - x^i} \\
&= \lim_{k \to \infty} \norm{x^{k + 1} - x^k},
  \end{split}
% https://math.stackexchange.com/questions/2064701/limit-of-average-of-sequence-elements
% https://math.stackexchange.com/questions/207910/prove-convergence-of-the-sequence-z-1z-2-cdots-z-n-n-of-cesaro-means
% https://math.stackexchange.com/questions/357693/inequality-concerning-limsup-and-liminf-of-cesaro-mean-of-a-sequence
\end{align}
where the above chain of equations follows \cite[Proof of Theorem 2.1]{Bailion1978} and the properties of the Ces\`aro summation for the last equation.

Hence,
$\lim_{k \to \infty} \norm{x^{k + 1} - x^k} = \lim_{k \to \infty}\norm{x^k - x^0}/k$, which is equal to $\norm{\ell}$, as we have shown in the first part of this proof. As a result, we also have $\lim_{k \to \infty} \norm{T x^k - x^k} = \norm{\ell}$ because $\epsilon^k \to 0$. 
We conclude that $\lim_{k \to \infty} T x^k - x^k = -\ell$ due to \cite[Lemma 2]{Pazy1971}. The desired $\lim_{k \to \infty}x^{k + 1} - x^k = -\ell$ then follows because $\epsilon^k \to 0$.
}
\newline

% \section{Results for the SDPLIB collection} \label{app:sdplib}
\newcontent{\noindent\rule{\textwidth}{0.3pt}}
\newcontent{We conclude the paper by providing} detailed results for the SDPLIB problems \newcontent{of} \S\ref{admm:subsec:sdplib}.
% References for the code of the following table
% https://tex.stackexchange.com/questions/63585/sidewaystable-together-with-longtable
% https://tex.stackexchange.com/questions/394188/problems-with-long-table-and-formatting-problems
% The manual for the csvsimple package and siunitx
% https://tex.stackexchange.com/questions/115195/table-captions-continued
% https://tex.stackexchange.com/questions/453107/how-to-use-longtable-with-resizebox
% https://stackoverflow.com/questions/3322563/make-latex-table-caption-same-width-as-table
\newcommand{\tableheadersdplib}{{Name} & {$n_{\max}$} &  {rank} & {$t^{\text{exact}}$}    & {Speedup}           & {$t_{\text{proj}}^{\text{exact}}$} & {$\text{Speedup}_{\text{proj}}$}  & {$\text{Iter}^{\text{exact}}$} & {Iter} & {$f^{\text{exact}}$}            & {$f$}           & {$f^*$} \\}
\begin{singlespace}
\begin{landscape}
  {\sisetup{round-mode=places, scientific-notation=true}
  {\small\tabcolsep=3pt
  \footnotesize
	\begin{longtable}{
		l
  *{2}{S[scientific-notation=false, round-precision=0, table-number-alignment=center]}
  *{1}{S[round-precision=2, table-format=1.2e-2]}
  *{1}{S[scientific-notation=false, round-precision=2, table-number-alignment=center]}
  *{1}{S[round-precision=2, table-format=1.2e-2]}
  *{1}{S[scientific-notation=false, round-precision=2]}
  *{2}{S[scientific-notation=false, round-precision=0, table-number-alignment=center, table-column-width=1.3cm]}
  *{3}{S[round-precision=2, table-format=1.2e-2]}
  }
  \captionsetup{width=1.5\textwidth}
  \caption{\footnotesize{Results for the SDPLIB collection (\S 6.1). $n_{\max}$ denotes the dimensions of the largest semidefinite variable at each problem, ``rank'' the maximum number of computed ritz pairs by LOBPCG in the last iteration of Algorithm \ref{admm:alg:lobpcg}, while $t^{\text{exact}}, t^{\text{exact}}_{\text{proj}}, \text{Iter}^{\text{exact}}, f^{\text{exact}}$ the solution time (in seconds), the time spent in projecting to the semidefinite cone, the ADMM iterations and the resulting objective when computing the projections exactly. $\text{Iter}$ and $f$ denote identical metrics when computing projections approximately. ``Speedup'' and ``$\text{Speedup}_{\text{proj}}$'' denote the speedups achieved for the full ADMM algorithm and only the projection part when using LOBPCG, and $f^*$ the optimal objective of each problem. Hardware used: Intel Gold 5120 with 192GB of memory.}}\label{admm:tab:sdplib}\\
  \toprule \tableheadersdplib \midrule \endfirsthead
  \caption*{Table \ref{admm:tab:sdplib}: Continued.}\\
	\toprule \tableheadersdplib \midrule \endhead
  \begin{filecontents*}{table_csv_data.tex}
ProblemName,MaxPsdDim,Rank,Time,TimeLOBCPG,Speedup,TimePrj,TimePrjLOBPCG,SpeedupPrj,Iter,IterLOBPCG,Objective,ObjectiveLOBPCG,OptimalObjective,Status,StatusLOBPCG
arch0,161.0,5.0,21.436373,21.4322090148925,1.0001942863241304,19.83755079,20.073206978,0.9882601625012747,2500.0,2500.0,235074.067,234848.458833057,0.566517,ITERATION_LIMIT,ITERATION_LIMIT
arch2,161.0,4.0,22.4147799,26.9344580173492,0.8321971760323539,20.82509295,25.5261977229999,0.8158321570641107,2500.0,2500.0,84869.5181,78838.5328938543,0.671515,ITERATION_LIMIT,ITERATION_LIMIT
arch4,161.0,5.0,22.42679095,34.309278011322,0.6536654878776289,20.46763633,32.87191917,0.622648048754009,2500.0,2500.0,329763.985,329581.814151575,0.9726274,ITERATION_LIMIT,ITERATION_LIMIT
arch8,161.0,6.0,22.25952792,34.9433329105377,0.6370178819802074,20.69167859,33.549525764,0.6167502555938661,2500.0,2500.0,1494044.75,1493059.15514071,7.05698,ITERATION_LIMIT,ITERATION_LIMIT
control10,100.0,26.0,36.09358597,36.9349908828735,0.9772193009186946,10.49403039,10.476712437,1.0016529949737707,2500.0,2500.0,730.928667,750.776714740206,38.533,ITERATION_LIMIT,ITERATION_LIMIT
control11,110.0,17.0,52.33573008,56.1203091144561,0.9325631113909666,12.45731898,12.971406674,0.9603676218840288,2500.0,2500.0,951.285986,827.700029745237,31.959,ITERATION_LIMIT,ITERATION_LIMIT
control6,60.0,16.0,7.916699886,10.4429340362548,0.7580915342867764,4.370442399,6.68445059399999,0.6538222307937978,2500.0,2500.0,644.053379,644.051124084778,37.3044,ITERATION_LIMIT,ITERATION_LIMIT
control7,70.0,15.0,11.04184604,15.3987638950347,0.7170605455909627,5.599818966,9.50095957199999,0.5893950946284487,2500.0,2500.0,679.065644,679.059669257117,20.6251,ITERATION_LIMIT,ITERATION_LIMIT
control8,80.0,14.0,16.32515407,21.3301119804382,0.7653571666652179,6.75430986,10.98055869,0.6151153188727231,2500.0,2500.0,697.073998,707.642579195875,20.286,ITERATION_LIMIT,ITERATION_LIMIT
control9,90.0,14.0,23.45479703,31.6211349964141,0.7417443122348335,8.068885871,14.208533487,0.5678901259150054,2500.0,2500.0,699.626723,847.85488840046,14.6754,ITERATION_LIMIT,ITERATION_LIMIT
equalG11,801.0,104.0,34.89896393,11.2118470668792,3.112686404106837,30.34134495,7.978323165,3.802972670135003,160.0,120.0,625.444318,625.322803599598,629.1553,OPTIMAL,OPTIMAL
equalG51,1001.0,258.0,105.126322,49.1700420379638,2.1380157031151774,92.68426105,42.6866852909999,2.171268638409402,280.0,160.0,3992.99727,3980.18389766675,4005.601,OPTIMAL,OPTIMAL
gpp100,100.0,8.0,1.266521215,0.402698993682861,3.1450816487448887,1.165204676,0.285850053,4.076279377146031,360.0,440.0,-44.8578565,-44.8628719103282,-44.9435,OPTIMAL,OPTIMAL
gpp124-1,124.0,8.0,3.003633022,0.659072160720825,4.557365947174793,2.760041339,0.446499111,6.181515866444804,520.0,520.0,-7.2819747,-7.28203042658788,-7.3431,OPTIMAL,OPTIMAL
gpp124-2,124.0,7.0,1.723228931,0.602716922760009,2.859101621220214,1.571182009,0.469742103999999,3.3447757729632923,320.0,320.0,-46.8145279,-46.8242437888415,-46.8623,OPTIMAL,OPTIMAL
gpp124-3,124.0,9.0,1.138211012,0.374979972839355,3.0353914727270523,0.93562171,0.250243397999999,3.738846728735692,200.0,280.0,-152.945925,-152.970857528746,-153.014,OPTIMAL,OPTIMAL
gpp124-4,124.0,9.0,2.240450144,0.309998989105224,7.227282096844246,2.085218564,0.199821486999999,10.435407099137493,360.0,240.0,-418.647414,-418.48043322116,-418.99,OPTIMAL,OPTIMAL
gpp250-1,250.0,10.0,13.715415,2.7282121181488,5.0272538959714215,12.33068197,1.500832549,8.215894556801754,720.0,720.0,-15.3948284,-15.3948660427332,-15.445,OPTIMAL,OPTIMAL
gpp250-2,250.0,12.0,18.84208488,2.81465697288513,6.6942739600293635,17.23328614,1.30444548399999,13.211196904262604,840.0,800.0,-81.8514974,-81.8497020984278,-81.869,OPTIMAL,OPTIMAL
gpp250-3,250.0,13.0,19.57245398,2.19988203048706,8.897047072868085,17.98653217,1.19021223299999,15.112037728484808,840.0,600.0,-303.460786,-303.441582769767,-303.5,OPTIMAL,OPTIMAL
gpp250-4,250.0,13.0,8.028266907,1.26345801353454,6.35420158089846,7.371153436,0.633095275,11.643039724155264,360.0,360.0,-747.128449,-747.05842636849,-747.3,OPTIMAL,OPTIMAL
gpp500-1,500.0,20.0,46.714674,13.4317519664764,3.477928576748044,42.37264106,8.621078531,4.915004649085941,560.0,680.0,-26.0694385,-26.0469870718337,-25.3,OPTIMAL,OPTIMAL
gpp500-2,500.0,19.0,54.621773,20.2196719646453,2.701417366983391,49.62376621,8.761304656,5.663969940369118,640.0,1720.0,-155.92478,-156.306780666579,-156.06,OPTIMAL,OPTIMAL
gpp500-3,500.0,21.0,81.65002608,5.65235090255737,14.44532151092353,74.37190381,2.481226722,29.9738444498342,960.0,440.0,-513.264076,-513.210975201251,-513.02,OPTIMAL,OPTIMAL
gpp500-4,500.0,22.0,27.12852311,4.9299988746643,5.502744280412695,24.22808701,2.249010099,10.772778219525462,320.0,360.0,-1567.57576,-1567.50437302578,-1567.02,OPTIMAL,OPTIMAL
maxG11,800.0,43.0,131.4779389,18.9518370628356,6.937477272735062,118.3520353,8.956438889,13.214184428295049,640.0,600.0,628.864234,628.868022341412,629.1648,OPTIMAL,OPTIMAL
maxG32,2000.0,70.0,1229.422218,136.443146944046,9.010509106068728,1113.006535,49.5356885959999,22.468780924343047,760.0,760.0,1566.63513,1566.63299581302,1567.64,OPTIMAL,OPTIMAL
maxG51,1000.0,54.0,72.24623704,22.7612280845642,3.174092222598246,64.32628613,12.261811169,5.246067260652985,200.0,360.0,3997.71402,3998.97062170513,4003.809,OPTIMAL,OPTIMAL
maxG55,5000.0,105.0,12697.66646,1415.01671481132,8.973509872420921,11517.6327,652.577885168,17.64943765606598,840.0,960.0,12809.7804,12812.4636960294,9999.21,OPTIMAL,OPTIMAL
maxG60,7000.0,178.0,46686.89131,4342.96085810661,10.750014295628251,43112.02436,2745.09906893699,15.705088697106538,880.0,920.0,15153.0047,15134.1487106395,15222.27,OPTIMAL,OPTIMAL
mcp100,100.0,8.0,1.368904829,0.421452045440673,3.2480678260053626,1.247576334,0.313371321999999,3.9811439222891107,280.0,320.0,226.161859,226.147991635571,226.1574,OPTIMAL,OPTIMAL
mcp124-1,124.0,20.0,2.884483099,1.2610740661621,2.2873225105473107,2.631275123,1.11098834799999,2.368409288663416,400.0,400.0,141.985003,141.984829946689,141.9905,OPTIMAL,OPTIMAL
mcp124-2,124.0,8.0,1.511134863,0.482953071594238,3.1289476180609284,1.376664937,0.322933401999999,4.262999517776747,240.0,240.0,269.879989,269.880212206375,269.8802,OPTIMAL,OPTIMAL
mcp124-3,124.0,9.0,1.492445946,0.464807987213134,3.2108870481084284,1.352068716,0.333171258999999,4.058179328127473,240.0,400.0,467.744822,467.743749752246,467.7501,OPTIMAL,OPTIMAL
mcp124-4,124.0,8.0,1.602807999,0.346652030944824,4.623679816995263,1.46905569,0.247928696,5.9253152769375275,200.0,240.0,864.420469,864.402925562451,864.4119,OPTIMAL,OPTIMAL
mcp250-1,250.0,30.0,16.81392908,6.24679112434387,2.6916105797864414,15.22061883,5.277233792,2.8842040034446894,680.0,680.0,317.262593,317.262600061794,317.2643,OPTIMAL,OPTIMAL
mcp250-2,250.0,14.0,11.69293213,1.73863101005554,6.72536729321679,10.49411303,0.989886341,10.601331279506967,560.0,520.0,531.919321,531.905278580937,531.9301,OPTIMAL,OPTIMAL
mcp250-3,250.0,13.0,9.157525063,1.78094887733459,5.1419359531001065,8.203635299,1.173922599,6.988225037995031,440.0,400.0,981.153388,981.145787533512,981.1726,OPTIMAL,OPTIMAL
mcp250-4,250.0,14.0,4.386127949,1.5135428905487,2.8979211467274046,3.932165615,0.888331239,4.42646328572939,200.0,360.0,1681.8768,1681.83775416399,1681.96,OPTIMAL,OPTIMAL
mcp500-1,500.0,72.0,98.9261179,78.5693531036377,1.259092941359852,91.09249493,60.6709869359999,1.50141772089666,920.0,2400.0,598.124002,597.721609103587,598.1485,OPTIMAL,OPTIMAL
mcp500-2,500.0,27.0,119.699671,9.65048503875732,12.403487546923708,108.7317951,4.92513884599999,22.07689945397333,1160.0,600.0,1069.6344,1069.73894265888,1070.057,OPTIMAL,OPTIMAL
mcp500-3,500.0,21.0,47.77815199,11.6502900123596,4.1010268361828714,42.50769592,5.858779185,7.255384539637672,400.0,760.0,1847.71995,1847.72653461,1847.97,OPTIMAL,OPTIMAL
mcp500-4,500.0,22.0,73.12671518,7.44562101364135,9.821439346163645,66.66909813,3.71700247099999,17.936253378939487,680.0,480.0,3566.42698,3566.49918215063,3566.738,OPTIMAL,OPTIMAL
qap10,101.0,29.0,1.212135077,0.585015058517456,2.0719724378920947,1.069287859,0.512159938999999,2.087800660644803,280.0,160.0,-1083.99079,-1075.3233598655,-10.93,OPTIMAL,OPTIMAL
qap8,65.0,20.0,0.370157957,0.327893018722534,1.1288985610066646,0.304369543,0.288270399,1.055847371273108,160.0,160.0,-742.085757,-742.256424330871,-757.0,OPTIMAL,OPTIMAL
qap9,82.0,24.0,0.48429203,0.418731927871704,1.1565681949822155,0.409884333,0.371827341999999,1.1023512439814098,160.0,160.0,-1389.12648,-1389.210948034,-1410.0,OPTIMAL,OPTIMAL
qpG11,1600.0,38.0,710.244102,124.463922977447,5.706425484665922,611.4170359,46.5273135069999,13.141034584083906,920.0,880.0,2448.35119,2448.44519035569,2448.659,OPTIMAL,OPTIMAL
qpG51,2000.0,20.0,3639.921227,359.365283966064,10.128750297827125,3203.217738,82.7550085109999,38.70723712842376,2500.0,2500.0,11833.6054,11852.2902318075,1181.0,ITERATION_LIMIT,ITERATION_LIMIT
ss30,294.0,8.0,97.01798105,101.713482141494,0.9538360009643356,88.71164897,96.0907098449999,0.9232073434892637,2500.0,2500.0,1944387.34,1943360.16159876,20.2395,ITERATION_LIMIT,ITERATION_LIMIT
theta2,100.0,18.0,10.40256715,3.89894795417785,2.668044629539953,9.562995813,3.385260162,2.824892432299861,2240.0,2240.0,32.9189051,32.9186021892303,32.87917,OPTIMAL,OPTIMAL
theta3,150.0,26.0,5.046710968,1.61847686767578,3.118185417903037,4.597507396,1.19506371199999,3.8470814148526653,440.0,560.0,42.2918008,42.2193586238576,42.16698,OPTIMAL,OPTIMAL
theta4,200.0,34.0,11.30239606,2.83579802513122,3.985613911793668,10.11370822,2.205851091,4.584946037955379,720.0,680.0,50.5001714,50.3733108338071,50.32122,OPTIMAL,OPTIMAL
theta5,250.0,17.0,11.64185715,13.1567730903625,0.8848565731157745,10.33283967,9.50005710200001,1.0876607960413909,480.0,2500.0,57.3636831,94.3479038082969,57.23231,OPTIMAL,ITERATION_LIMIT
theta6,300.0,13.0,16.75213313,16.4818868637084,1.0163965611781172,14.59007889,10.934850861,1.334273240253935,480.0,2500.0,63.6295286,-2.39071739287342,63.47709,OPTIMAL,ITERATION_LIMIT
thetaG11,801.0,26.0,771.0980201,37.091989994049,20.788801577475734,737.0046382,15.382163023,47.912938973407215,1280.0,1280.0,399.018015,399.015452324797,400.0,OPTIMAL,OPTIMAL
thetaG51,1001.0,116.0,1307.13365,241.848583936691,5.404760403071742,1194.486428,170.657446238,6.999322059080629,2500.0,2500.0,339.923028,343.049420310537,349.0,ITERATION_LIMIT,ITERATION_LIMIT
\end{filecontents*}
\csvreader[head to column names]{table_csv_data.tex}{}
  {\\
  \ProblemName	& \MaxPsdDim & \Rank & \Time	& \Speedup	& \TimePrj & \SpeedupPrj	& \Iter	& \IterLOBPCG	& \Objective & \ObjectiveLOBPCG & \OptimalObjective
  }
  \\\hline
  \end{longtable}}
  }
\end{landscape}
\end{singlespace}

\end{document}